\documentclass[microtype]{gtpart}     % Basic GT/GTM/AGT style
\usepackage{amsfonts} 
\usepackage{epsfig} 
\usepackage{color}
\usepackage{pinlabel}
\usepackage{amscd}
\usepackage[all]{xy}
\usepackage{pinlabel}
\usepackage{mathrsfs}  
\usepackage{tcolorbox}
\usepackage{hyperref}
\usepackage{float}
\usepackage{tikz}
\usepackage{geometry}
\usepackage{tikz-cd}
\usepackage{mathtools}
\usepackage{enumitem}

\usepackage{comment}
\usepackage{caption}

\title{L-spaces, taut foliations and fibered hyperbolic two-bridge links}

\author{Diego Santoro}
\givenname{Diego}
\surname{Santoro}
\address{University of Vienna, Oskar-Morgenstern-Platz 1, 1090 Vienna}
\email{diego.santoro95@gmail.com}
\newtheorem{lemma}{Lemma}[section]
\newtheorem{teo}[lemma]{Theorem}
\newtheorem{prop}[lemma]{Proposition} 
\newtheorem{cor}[lemma]{Corollary} 
\newtheorem{conj}[lemma]{Conjecture}

\theoremstyle{definition}
\newtheorem{defn}[lemma]{Definition}

\newtheorem{example}[lemma]{Example}

\theoremstyle{remark}
\newtheorem{rem}[lemma]{Remark}
\newcommand{\matR} {\ensuremath {\mathbb{R}}}
\newcommand{\matQ} {\ensuremath {\mathbb{Q}}}
\newcommand{\matZ} {\ensuremath {\mathbb{Z}}}

\newcommand{\matD} {\ensuremath {\mathbb{D}}}

\newcommand{\lk} {\ensuremath {{\rm lk}}}

\begin{document}

\begin{abstract}    % type your abstract below

We prove that if $M$ is a rational homology sphere that is Dehn surgery on a fibered hyperbolic two-bridge link, then $M$ is not an $L$-space if and only if $M$ supports a co-orientable taut foliation. As a corollary we show that if $K'$ is obtained by a non-trivial knot $K$ as result of an operation called \emph{two-bridge replacement}, then all non-meridional surgeries on $K'$ support co-orientable taut foliations. This operation generalises Whitehead doubling and as a particular case we deduce that all non-meridional surgeries on Whitehead doubles of a non-trivial knot support co-orientable taut foliations.
\end{abstract}

\maketitle

%%%%%%%%%%%%%%%%%%%%   Start of main body of article

\section{Introduction}
In recent years the field of low-dimensional topology has seen a growing interest in the study of the so-called $L$-space conjecture. This conjecture predicts that the following notions of ``complexity" are all equivalent:

\begin{conj}[$L$-space conjecture]\label{L space conjecture} 
For an irreducible oriented rational homology $3$-sphere $M$, the following are
equivalent:
\begin{enumerate}
    \item  $M$ supports a co-oriented taut foliation;
    \item  $M$ is not an $L$-space, i.e. its Heegaard Floer homology is not minimal;
    \item  $M$ is left-orderable, i.e. $\pi_1(M)$ is left-orderable.
\end{enumerate}
\end{conj} 
The equivalence between $(1)$ and $(2)$ was conjectured by Juh\'asz in \cite{J}, while the equivalence between $(2)$ and $(3)$ was conjectured by Boyer, Gordon and Watson in \cite{BGW}.
This conjecture predicts strong connections among geometric, dynamical,
Floer homological, and algebraic properties of 3-manifolds. Despite its boldness, as a result of the work by many researchers \cite{BC,BGW,BNR,CLW,EHN,HRRW,LS} it is now known that the conjecture holds for all the graph manifolds, i.e. the manifolds whose JSJ decomposition includes only Seifert fibered pieces. Moreover the results of Oszváth-Szab\'o \cite{OS}, Bowden \cite{Bow} and Kazez-Roberts \cite{KR2} imply that in general manifolds supporting co-orientable taut foliations are not $L$-spaces.
\\

A natural way to investigate this conjecture is by using Dehn surgery descriptions of $3$-manifolds. For instance, it is known that if a non-trivial knot $K$ in $S^3$ has a positive surgery that is an $L$-space, then $K$ is prime \cite{Krca}, fibered \cite{G,Ni} and strongly quasipositive \cite{Hedden}. Moreover, the $r$-framed surgery on such a knot $K$ is an $L$-space if and only if $r\geq 2g(K)-1$, where $g(K)$ denotes the genus of $K$ \cite{KMOS}. Taut foliations on manifolds obtained as surgery on knots in $S^3$ are constructed for example in \cite{R,R1,DR1,DR2, K} and it is possible to prove the left-orderability of some of these
manifolds by determining which of these foliations have vanishing Euler class, as done in \cite{H}. Another approach to study the left-orderability of surgeries on knots is via representation theoretic methods, as presented in \cite{CuD} and \cite{DuR}.
\\

When it comes to investigate surgeries on links, the story becomes more mysterious. For instance there is no generalisation of the result of \cite{KMOS} we cited in the previous paragraph -- even if it holds in some cases, see for example \cite[Lemma~2.6]{S} -- and links admitting $L$-space surgeries need not to be fibered \cite[Example~3.9]{Liu} nor quasipositive \cite[Proposition~1.5]{BeiCav}. Concerning foliations, in \cite{KR1} Kalelkar and Roberts construct co-orientable taut foliations on some fillings of $3$-manifolds that fiber over the circle and in particular their methods can also be applied to surgeries on fibered links. In \cite{S}, taut foliations on all the surgeries on the Whitehead link that are not $L$-spaces are constructed.

In this paper we study the $L$-space conjecture for manifolds that can be obtained as surgery on two-bridge links. A two-bridge link is either hyperbolic or isotopic to the $(2n,2)$ torus link, for some integer $n$. In the latter case the exterior is a Seifert fibered manifold and since the conjecture has been proven for graph manifolds -- in particular, see \cite{EHN, JN, Naimi, LSIII, BGW} for the case of Seifert fibered manifolds -- we focus our study on hyperbolic two-bridge links. 
The main theorem of this paper is the following:

\begin{teo}\label{thm: main theorem}
Let $L$ be a fibered hyperbolic two-bridge link and let $M$ be a manifold obtained as Dehn surgery on $L$. Then $M$ admits a co-orientable taut foliation if and only if $M$ is not an $L$-space.
\end{teo}

\begin{rem}
In contrast to the case of knots, the property of being fibered for a link depends on the choice of an orientation of the link. This happens for instance in the case of the $(2n,2)$ torus link for $n>1$, see for example \cite[Example~3.1]{BaaGra}. On the other hand, changing orientations of the components of $L$ has no effects on the study of the $L$-space conjecture for the surgeries on $L$. For this reason we will consider links as unoriented and say that a link is fibered if there exists an orientation for which it is a fibered link.
\end{rem}

\begin{rem}
Theorem \ref{thm: main theorem} is a result about the exterior of the links.
Links are not uniquely determined by their complement, hence it can be a priori possible that many non-isotopic hyperbolic two-bridge links share the same exterior. However it follows from \cite[Theorem~1.4]{MW} that the exteriors of hyperbolic two-bridge links (with two components) are not even commensurable.
\end{rem}

We will be able to completely determine for each fibered hyperbolic two-bridge link $L$ the set of surgeries on $L$ that are $L$-spaces. We denote by $\mathcal{L}(L)$ the set of slopes on $L$ that produce $L$-spaces. Recall that since $L$ is a link in $S^3$ there is a canonical identification between the set of slopes on $L$ and $\overline{\matQ}\times \overline{\matQ}$, where $\overline{\matQ}=\matQ\cup \{\infty\}$, obtained by considering on each component of $L$ its canonical meridian and longitude basis. Also, notice that we can reduce our study to the rational surgeries. In fact the components of two-bridge links are unknotted, so when one of the two surgery coefficients is infinite the only rational homology spheres that can be obtained are $S^3$ and lens spaces. We will prove the following proposition, where the link $L_n$ is shown in Figure \ref{figure:link L_n}, for $n\geq 1$.

\begin{prop}\label{prop: link L_n}
Let $L$ be a fibered hyperbolic two-bridge link. Then
\begin{itemize}
\item if $L$ is isotopic as an unoriented link to $L_n$, then $\mathcal{L}(L)\cap \matQ^2=\big{(}[n, +\infty)\times [n, +\infty)\big{)}\cap \matQ^2$;
\item if $L$ is isotopic as an unoriented link to the mirror of $L_n$, then $\mathcal{L}(L)\cap \matQ^2=\big{(}(-\infty,-n]\times (-\infty, -n]\big{)}\cap \matQ^2$;
\item if $L$ is not isotopic as an unoriented link to any of the links $L_n$ or their mirrors, then $\mathcal{L}(L)\cap \matQ^2=\emptyset$.
\end{itemize}
\end{prop}

We observe that $L_1$ is the Whitehead link. The $L$-space conjecture for surgeries on the Whitehead link was studied by the author in \cite{S}. As a consequence of the previous proposition we have the following Dehn surgery characterisation of the Whitehead link:

\begin{cor}
Let $L$ be a fibered hyperbolic two-bridge link and suppose that the $(1,1)$-surgery on $L$ is an $L$-space. Then $L$ is isotopic, as unoriented link, to the Whitehead link.
\end{cor}

\begin{figure}[]
    \centering
    \includegraphics[width=0.8\textwidth]{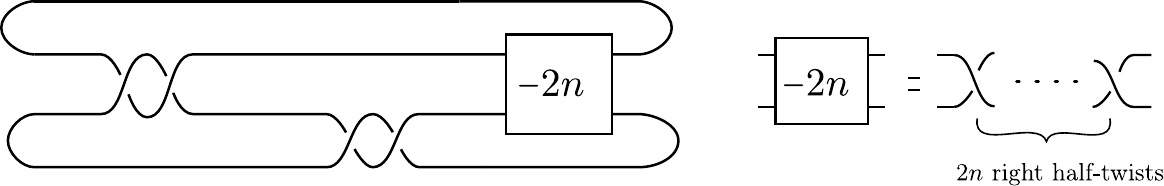}
    \caption{The link $L_n$.}
    \label{figure:link L_n}
\end{figure}

We observe that all the links $\{L_n\}_{n\geq 1}$ can be obtained as surgery on a $3$-component link, see Figure \ref{figure:Ln+L}. On the other hand we have the following:
\begin{prop}
It is not possible to obtain the exteriors of all the hyperbolic fibered two-bridge links as Dehn filling on a fixed cusped hyperbolic manifold $N$. In particular there exists no hyperbolic link $L$ such that every hyperbolic fibered two-bridge link is surgery on $L$.
\end{prop}
\begin{proof}
By using the main result of \cite{Lackenby}, it follows that there exists a family of fibered hyperbolic two-bridge links whose volumes grow to infinity (this is the family of links associated to $L(a_1, \dots, a_n)=L(2, 2, \dots, 2$) in the notation introduced in Section \ref{sec: L-spaces}). Volume decreases under hyperbolic Dehn filling \cite{Thurston}, hence we obtain the thesis.
\end{proof}

As two-bridge links have tunnel number one, all surgeries on these links have at most Heegaard genus two. It has been proven by Li in \cite{L1} that if a closed orientable irreducible three manifold with Heegaard genus two has left-orderable fundamental group, then it admits a co-orientable taut foliation. As a consequence of this result together with Proposition \ref{prop: link L_n} we have:

\begin{cor}
Let $M$ be obtained as $(r_1, r_2)$-surgery on the link $L_n$, with $(r_1, r_2)\in [n, +\infty)\times [n, +\infty)$ and suppose that $M$ is irreducible. Then $M$ is not left-orderable.
\end{cor}

\textbf{Applications to satellites on knots and links.}
We briefly recall the satellite operation. Suppose that $P$ is a knot inside the standard solid torus $V=\matD^2\times S^1$ and assume that $P$ is not contained in a $3$-ball of $V$. Let $K$ be a knot in $S^3$ and let ($\mu_K$, $\lambda_K$) be a meridian-longitude basis of $K$. Consider the orientation preserving diffeomorphism $\phi$ between $V$ and a tubular neighbourhood of $K$ mapping, as oriented curves, the meridian $\mu_V$ and longitude $\lambda_V$ of $V$ to $\mu_K$ and $\lambda_K$ respectively.
The image of $P$ under $\phi$ is a knot $S$ in $S^3$, called a \emph{satellite} of $K$. The knot $K$ is called the \emph{companion} of $S$ and the knot $P$ is called the \emph{pattern} of $S$.

Let $L$ be a fibered hyperbolic two-bridge link, let denote by $\mathcal{K}$ one of its component and orient it arbitrarily. Since two-bridge links have unknotted components, the exterior of $\mathcal{K}$ in $S^3$ is a solid torus $V$ and we can use the other component as pattern $P$ for producing satellite knots. 
We also fix a meridian-longitude basis for $V$ given by $(\mu_V, \lambda_V)=(\lambda_{\mathcal{K}}, \mu_{\mathcal{K}})$, where $\mu_{\mathcal{K}}$ and $\lambda_{\mathcal{K}}$ are the canonical meridian and longitude of $\mathcal{K}$. 

\begin{example}
If $L$ is the Whitehead link we obtain the Whitehead pattern. This is the pattern used to define Whitehead doubles of knots, see Figure \ref{figure:Whitehead pattern}. 
\end{example}

We now define an operation, that we call \emph{two-bridge replacement}, that generalises Whitehead doubling.
\begin{defn}
Let $K$ be a knot in $S^3$ and let $L=\mathcal{K}\sqcup P$ be a fibered hyperbolic two-bridge link. A knot $K'$ in $S^3$ is a \emph{two-bridge replacement} of $K$ if $K'$ is a satellite knot of $K$ with pattern $P$. 
More generally, if $\mathcal{L}=K_1\sqcup\cdots \sqcup K_d$ is a link with $d$ components, we say that $\mathcal{L}'=K_1'\sqcup \cdots \sqcup K_d'$ is a \emph{two-bridge replacement} of $\mathcal{L}$ if each knot $K_i'$ is a two-bridge replacement of $K_i$, for $i=1, \dots, d$.
\end{defn}

Notice that in a two-bridge replacement of a link $\mathcal{L}$ it is allowed to act on different components on $\mathcal{L}$ by two-bridge replacement with different two-bridge links.
We also remark that in the definition of two-bridge replacement we ask $L$ to be fibered and hyperbolic.

\begin{figure}[]
    \centering
    \includegraphics[width=0.3\textwidth]{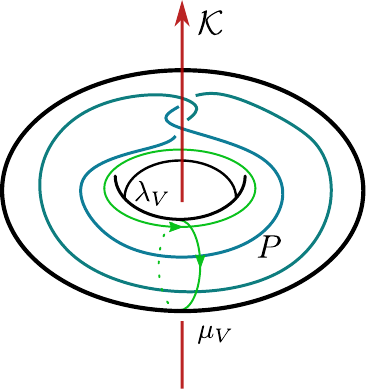}
    \caption{The (positive clasped) Whitehead pattern. The meridian $\mu_V$ is given by the longitude of the knot $K_0$ and the longitude $\lambda_V$ by its meridian. By considering the mirror of the Whitehead link one obtains the negative clasped Whitehead pattern.}
    \label{figure:Whitehead pattern}
\end{figure}

Recall that the genus of a link $\mathcal{L}$ is the minimal genus of a connected Seifert surface for $\mathcal{L}$, and that if $\mathcal{L}$ is fibered (see Definition \ref{def: fibered link}) then a Seifert surface has minimal genus if and only if it is a fibered surface \cite[Chapter~1.4]{EN}. 
The proofs of Theorem \ref{thm: main theorem} and of the main theorem of \cite{S}, together with results from \cite{KR1} and \cite{LR}, imply the following theorem. 
\begin{teo}
Let $\mathcal{L}$ be any non-trivial knot, or any fibered link with positive genus and let $\mathcal{L}'$ denote a two-bridge replacement of $\mathcal{L}$. Then all manifolds obtained by doing surgery on each component of $\mathcal{L}'$ along a non-meridional slope support a co-orientable taut foliation.
\end{teo}
\begin{proof}
\begin{itemize}[leftmargin=*]
\item We first analyse the case where $\mathcal{L}$ is a non-trivial knot. We denote it by $K$ and we denote by $K'$ its two-bridge replacement. We use the notation introduced above, and so we denote by $L=\mathcal{K}\sqcup P$ the fibered hyperbolic two-bridge link used in the definition of two-bridge replacement, and $V$ is the exterior of $\mathcal{K}$ in $S^3$. Recall that $V$ is a solid torus. We also denote by  $E_{K'}$, $E_K$ and $E_L$ the exteriors in $S^3$ of $K', K$ and $L$ respectively. Let $\phi$ be the map from $V$ to a tubular neighbourhood of $K$ given by the definition of the satellite operation.
Our aim is to prove that all non-trivial surgeries on $K'$ support a co-orientable taut foliation. The key observation is that $E_{K'}$ is obtained by gluing $E_K$ and $E_L$. More precisely we have
$$
E_{K'}\cong E_K\cup_\varphi E_L
$$
where $\varphi$ is the restriction of $\phi$ to $\partial V\subset \partial E_L$.

In order to construct foliations on $E_{K'}$ we can therefore reduce ourselves to construct foliations on $E_K$ and $E_L$ and to study the gluing map $\varphi$.
We fix the canonical meridian-longitude basis $(\mu_K, \lambda_K)$ for the knot $K$ and we use it to identify slopes on $K$ with $\overline{\matQ}$. By definition $\phi$
maps the meridian $\mu_V$ of $V$ to $\mu_K$ and the longitude $\lambda_V$ of $V$ to some longitude of $K$, i.e.
\begin{align*}
    &\phi(\lambda_{\mathcal{K}})=\phi(\mu_V)=\mu_K\\
    &\phi(\mu_{\mathcal{K}})=\phi(\lambda_V)=l\mu_K+\lambda_K
\end{align*}
for some integer $l\in \matZ$.
Given two coprime integers $p,q$ the map $\phi$ satisfies
$$
p\mu_K+q\lambda_K=\phi((p-ql)\lambda_{\mathcal{K}}+q\mu_{\mathcal{K}})
$$
and therefore its restriction $\varphi$ acts on the slopes by identifying the slope $\frac{p}{q}$ on $K$ with the slope $(\frac{p}{q}-l)^{-1}$ on $\mathcal{K}$.
Since $K$ is a non-trivial knot, it follows by \cite[Theorem~1.1]{LR} that there exists an interval $(-a, b)$, where $a,b>0$, such that for every slope $s\in (-a, b)$ there exists a co-orientable taut foliation on $E_K$ intersecting the boundary torus in a collection of circles of slope $s$. 
The slope $0$ on $K$ corresponds, via the identification given by $\varphi$, to the slope $-\frac{1}{l}$ on $\mathcal{K}$. Hence the interval $(-a,b)$ is identified with a neighbourhood $U_1$ of $-\frac{1}{l}\subset \overline{\matQ}$.  The proof of \cite[Theorem~1.1]{S} when $L$ is the Whitehead link and the proof of Theorem \ref{thm: main theorem} when $L$ is any other fibered hyperbolic two-bridge link imply the following fact: for every integer $l\in \matZ$ and every neighbourhood $U$ of $-\frac{1}{l}\subset \overline{\matQ}$ there exists a slope $r\in U$ such that for every non-meridional slope $r'$ on $P$ there is a co-orientable taut foliation on $E_L$ intersecting the boundary tori in circles of slopes $r$ and $r'$ respectively. 

By choosing a slope $ r\in U_1$ guaranteed by the previous observation we are able to find for each non-meridional slope $r'$ in $E_{K'}$ taut foliations $\mathcal{F}$ on $E_K$ and $\mathcal{F}'$ on $E_L$ that can be glued along $\varphi$ to define a co-orientable taut foliation in $E_{K'}$ intersecting the boundary in parallel curves of slope $r'$. By capping off with meridional discs, these foliations extend to the surgeries on $K'$.
\item When $\mathcal{L}=K_1\sqcup \dots \sqcup K_d$ is a fibered link with multiple components and positive genus we can proceed in an analogous way. Let $S$ denote the fiber surface for $\mathcal{L}$. By intersecting $S$ with the boundaries of tubular neighbourhoods of the knots $K_1,\dots, K_d$ we obtain longitudes $\lambda^S_1,\dots \lambda^S_d$. We use them to define meridian-longitude bases for the components of $\mathcal{L}$ and to identify slopes on the exterior of $\mathcal{L}$ with $\overline{\matQ}^d$.
It follows by \cite[Theorem~1.1]{KR1} that for every multislope $(r_1,\dots, r_d)$ in a neighbourhood of $0\in \overline{\matQ}^d$ there exists a co-orientable taut foliation in the exterior of $\mathcal{L}$ intersecting the boundary tori in parallel curves of slopes $r_1,\dots, r_d$ respectively. The statement now follows by applying to each component of $\mathcal{L}$ the same reasoning as in the previous case, where we never made use of the fact that $\lambda_K$ was the canonical longitude of $K$.
\end{itemize}
This concludes the proof.
\end{proof}

Two-bridge replacement generalises Whitehead doubling and we emphasise the following corollary.

\begin{cor}
Let $K$ be a non-trivial knot and let $K'$ be any Whitehead double of $K$. Then all non-trivial surgeries on $K'$ support a co-orientable taut foliation.\qed
\end{cor}

\textbf{Structure of the paper.} In Section \ref{sec: two-bridge background} we recall some notions on two-bridge links and describe some properties of fibered two-bridge links that will be used in the subsequent sections.
Section \ref{sec:taut foliations} is devoted to the construction of taut foliations and will take most of the paper. In Section \ref{subsec: background} we introduce branched surfaces and recall some of their basic properties, together with the main result of \cite{L2}. In Section \ref{subsec: constructing branched surfaces} we recall a general method of constructing branched surfaces in fibered manifolds with boundary. In Section \ref{subsec: boundary train tracks} we briefly discuss the boundary train tracks of these branched surfaces and from Section \ref{subsec: Fibered hyp two-bridge links} we start focusing our attention on surgeries on fibered hyperbolic two-bridge links: we subdivide them in four families that we analyse in Sections \ref{subsec: family 1}--\ref{subsec: family 4}.
In Section \ref{sec: L-spaces}  we recall part of the main result of \cite{RR} and then we use it to study the $L$-space surgeries on the links $\{L_n\}_{n\geq 1}$. 
\newline

\textbf{Acknowledgments.} I warmly thank my advisors Bruno Martelli and Paolo Lisca for their support and for their useful comments on a first draft of the paper. I also thank the anonymous referees for their comments and suggestions, which improved the clarity and overall organization of the paper. The author is supported by the FWF project P 34318 ``Cut and Paste Methods
in Low Dimensional Topology''.

\section{Background on two-bridge links}\label{sec: two-bridge background}

In this section we briefly recall the definition of fibered link and some facts about two-bridge links that we will often use in the paper. We refer to \cite{BZ} for proofs and details regarding two-bridge links.

\begin{figure}[]
    \centering

    \includegraphics[width=0.75\textwidth]{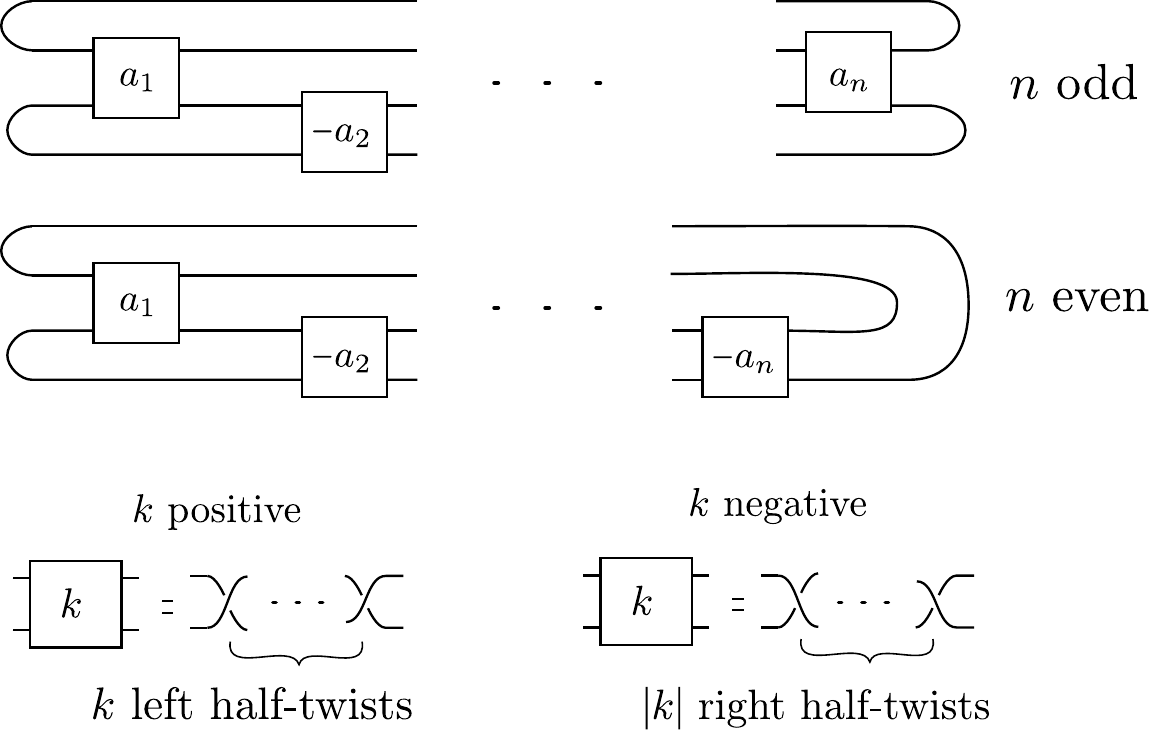}
    \caption{The two-bridge knot or link $L(a_1,\dots, a_n)$.}
    \label{figure:two bridge}
\end{figure}

A \emph{two-bridge link} can be described by a rational number $\frac{p}{q}$, where $p$ and $q$ are coprime integers, $p>0$, $q$ is odd and $0<|q|<p$, in the following way. We fix a sequence of integers $(a_1,\dots,a_n)$ such that
\begin{equation*}\label{eq:continued fraction}
\frac{p}{q}=a_1+\cfrac{1}{a_2+\cfrac{1}{\ddots+\cfrac{1}{a_n}}}                 \tag{*}
\end{equation*}
and consider the link defined by the diagram in Figure \ref{figure:two bridge}. We denote this link by $L(a_1,\dots, a_n)$.

We are interested in the case when $L(a_1,\dots, a_n)$ has two components. This happens exactly when the fraction $\frac{p}{q}$ has numerator $p$ even \cite[Proposition~12.3]{BZ}. 
When $L(a_1,\dots, a_n)$ is a link we orient the components as in Figure \ref{fig: two bridge ori}.

\begin{figure}
    \centering
    \includegraphics[width=0.8\textwidth]{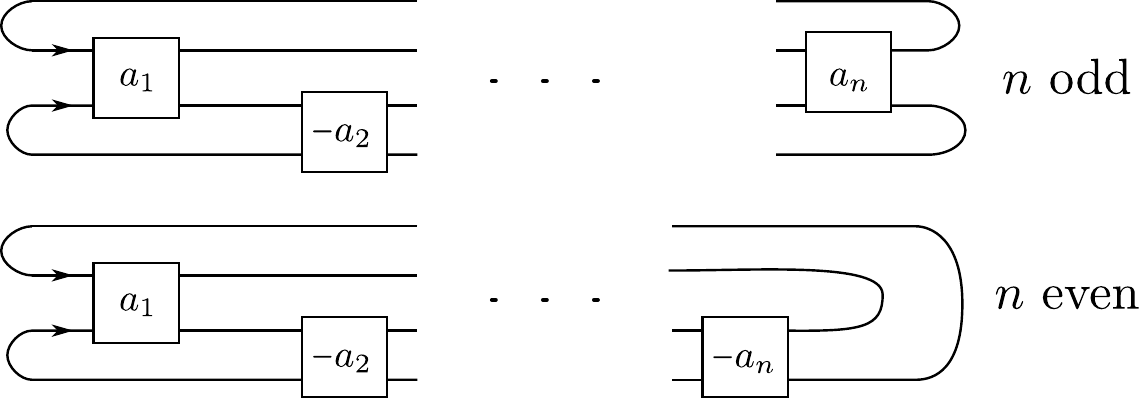}
    \caption{The oriented two-bridge link $L(a_1,\dots, a_n)$.}
    \label{fig: two bridge ori}
\end{figure}

A priori it could happen that the isotopy class of the two-bridge link associated to $\frac{p}{q}$ depends on the choice of the continued fraction representation of $\frac{p}{q}$. This is not the case, by the following theorem.

\begin{teo}[\cite{Schubert}, see also \cite{BZ}]\label{teo: classification two-bridge links}
Let $L=L(a_1\dots,a_n)$ and $L'=L(b_1,\dots,b_m)$ be two oriented two-bridge links and let $\frac{p}{q}$ and $\frac{p'}{q'}$ be the rational numbers defined as in \eqref{eq:continued fraction}. Then the links $L$ and $L'$ are isotopic if and only if $p=p'$ and $q'\equiv q^{\pm 1 } \mod 2p$. If $p=p'$ and $q'\equiv q+p \mod 2p$ or $qq'\equiv 1+p \mod 2p$,
then $L$ and $L'$ are isotopic after reversing the orientation of one of the components.
\end{teo}
We denote by $b(p,q)$ the two-bridge link associated to the rational number $\frac{p}{q}$. We also recall, for later reference, the following fact \cite[Exercise~2.1.16]{Kaw}.

\begin{lemma}\label{lemma: symmetric}
Two-component two-bridge links are symmetric, i.e. there exists an ambient isotopy of $S^3$ interchanging their components.\qed
\end{lemma}
We are interested in studying fibered hyperbolic two-bridge links. For this reason we recall the definition of fibered link.

Given an oriented surface $S$ with boundary and an orientation-preserving homeomorphism $h:S\rightarrow S$ fixing $\partial S$ pointwise, we denote by $M_h$ the mapping torus of $h$
$$
M_h=\frac{S\times [0,1]}{(h(x),0)\sim (x,1)}.
$$
We orient $S\times [0,1]$ as a product and $M_h$ with the orientation induced by $S\times [0,1]$. We also identify $S$ with its image in $M_h$ via the map
\begin{align*}
S&\rightarrow S\times \{0\}\subset M_h\\
x&\mapsto (x,0). 
\end{align*}
The homeomorphism $h$ is called the \emph{monodromy} of $M_h$

\begin{defn}\label{def: fibered link}
Let $L$ be an oriented link in $S^3$. We say that $L$ is \emph{fibered} if there exists a Seifert surface $S$ for $L$, an orientation preserving homeomorphism $h$ of $S$ fixing $\partial S$ pointwise and an orientation preserving homeomorphism
$$
\chi: S^3\setminus {\rm int}(N_L) \rightarrow M_h,
$$
where $N_L$ denotes a tubular neighbourhood of $L$ in $S^3$, so that 
\begin{itemize}
    \item $\chi_{|S}$ is the inclusion $S\subset M_h$;
    \item $\chi(m_i)=\{x_i\}\times [0,1]$, where $m_i$ is a meridian for the $i$-th component of $L$ and $x_i\in \partial S$ is a point.
\end{itemize}
\end{defn}

In the case of fibered links, we refer to the homeomorphism $h$ as the \emph{monodromy of the link}.

The following lemma yields us an useful characterisation of fibered hyperbolic two-bridge links.

\begin{lemma}\label{lemma: characterisation of fib hyp}
Let be $L$ a two-bridge link with two components. Then $L$ is fibered if and only $L=L(2b_1,\dots, 2b_n)$ as an unoriented link, where $|b_i|=1$ for all $i$'s and $n$ is odd. Moreover $L$ is fibered and hyperbolic if and only if it admits such a description with $(b_1,\dots,b_n)\ne\pm(1, -1, 1 ,\dots, (-1)^{n-1})$. 
\end{lemma}
\begin{proof}
It follows from \cite[Exercise~2.1.14]{Kaw} that any two-bridge link $L$ with two components can be written as $L=L(2b_1,\dots, 2b_n)$, where $b_i$ is a non-zero integer and $n$ is odd. Moreover it follows from \cite[Proposition~2]{GK}\footnote{The proof presented there is for knots, but the same proof works also for links.} that $L$ is fibered if and only if we can find such a description with $|b_i|=1$ for all $i$. This proves the first part of the lemma.
Since two-bridge links are non-split, prime, alternating links (see \cite{BZ}), as a consequence of \cite[Corollary~2]{M}, a two-bridge link is hyperbolic if and only if is it not a torus link. The only non-trivial torus links that are two-components two-bridge links are the $(2m, 2)$ torus links, with $m$ a non-zero integer, and by using Theorem \ref{teo: classification two-bridge links} one can see that $L$ is a torus link if and only if any description of $L$ as $L(2b_1,\dots, 2b_n)$ with all $|b_i|=1$ satisfies $(b_1,\dots,b_n)=\pm(1, -1, 1 ,\dots, (-1)^{n-1})$.
\end{proof}

Suppose $L=L(2b_1,\dots, 2b_n)$ where $|b_i|=1$ for all $i$'s and $n=2k+1$ is odd. Then it is possible to draw an explicit fiber surface $S$ for $L$. This surface is obtained by starting with the boundary connected sum of $k$ Hopf bands, and then plumbing other $k+1$ Hopf bands to this surface. The sign of the Hopf bands is determined in a straightforward way from the coefficients $(b_1,\dots, b_n)$. One example is described in Figure \ref{fig:plumbing two bridge}.  We also fix an orientation of $S$, so that in the figure the positive side is coloured in pink, and this induces an orientation of the link.

\begin{figure}[]
    \centering
    \includegraphics[width=0.65\textwidth]{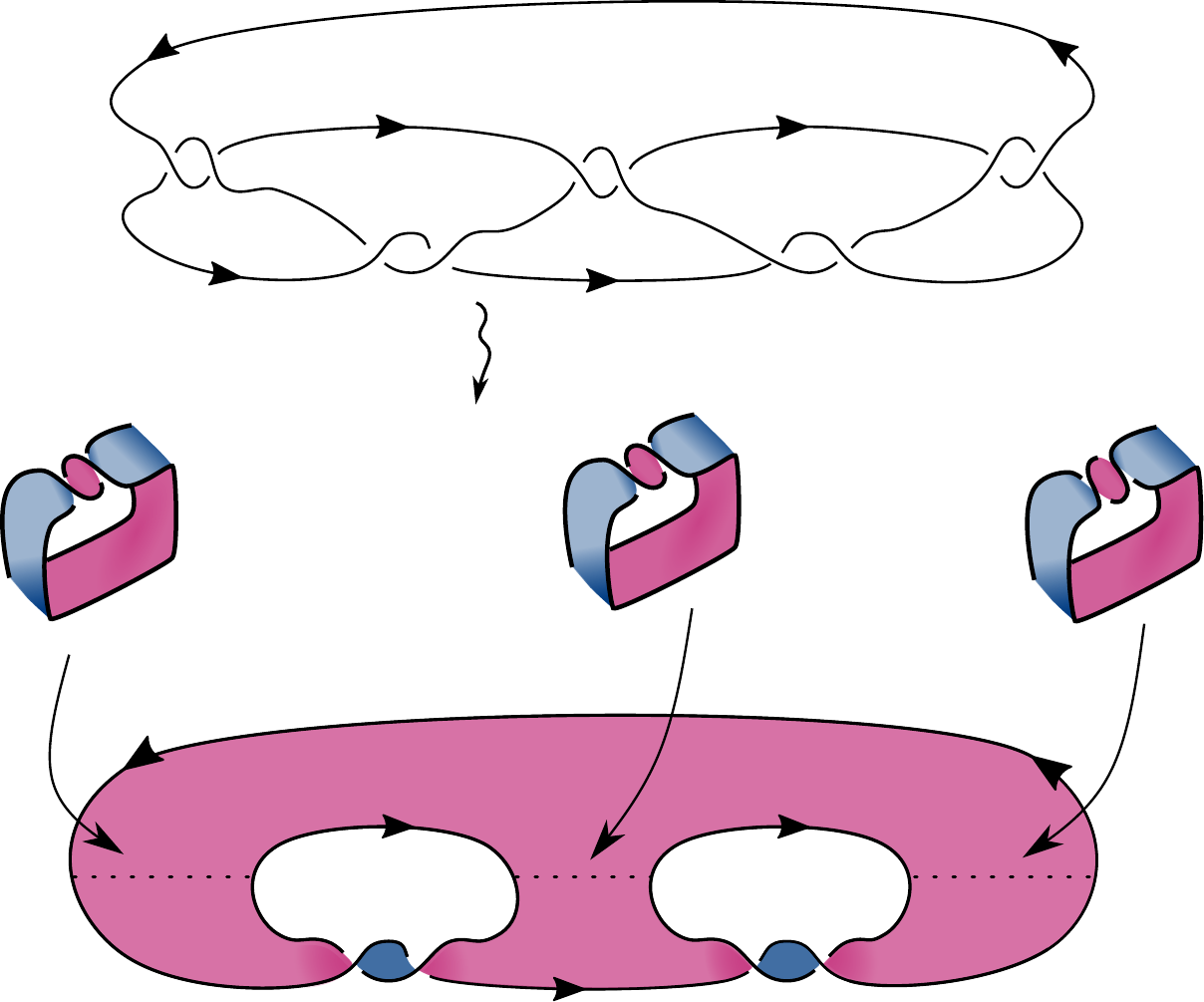}
    \caption{The fiber surface of the link $L(-2, -2, -2, 2, 2)$. The positive side is coloured in pink.}
    \label{fig:plumbing two bridge}
\end{figure}

From this very easy description of the fiber surface of $L$ we are able to determine the monodromy of $L$. More precisely, $S$ can be described in a more abstract way as in Figure \ref{fig:curves} and the monodromy is given by the diffeomorphism (to be read from right to left) 
\begin{equation*}\label{eq: monodromia}
    h=\tau_2^{\varepsilon_2}\tau_4^{\varepsilon_4}\dots\tau_{2k}^{\varepsilon_{2k}}\tau_1^{\varepsilon_1}\tau_3^{\varepsilon_3}\dots\tau_{2k+1}^{\varepsilon_{2k+1}}
    \tag{$\ast$}
\end{equation*}

where $\tau_i$ denotes the positive (i.e. the right) Dehn twist along the curve $\gamma_i$ shown in Figure \ref{fig:curves} and $$
\varepsilon_i=
\begin{cases}
-{\rm sgn}(b_i) &\text{ when $i$ is even}\\
{\rm sgn}(b_i) &\text{ when $i$ is odd}
\end{cases}.
$$
This follows from the fact that the monodromy of the boundary of a positive (resp. negative) Hopf band is a positive (resp. negative) Dehn twist along its core and from the way the monodromy of a plumbing or a boundary connected sum (or more generally a Murasugi sum), behaves with respect to the monodromies of the summands, see \cite[Corollary~1.4]{Gabai1}. 

\begin{figure}[]
    \centering
    \includegraphics[width=0.65\textwidth]{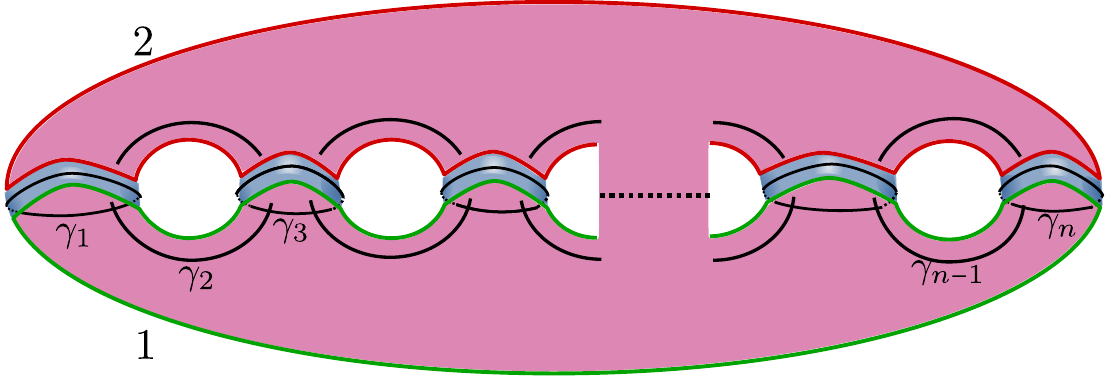}
    \caption{An abstract drawing of the fiber surface $S$ together with the curves $\gamma_i$'s. We have also coloured the two boundary components of $S$.}
    \label{fig:curves}
\end{figure}

To ease the exposition of some lemmas in the next section we fix the following notation. With reference to Figure \ref{fig:curves} we say that a Dehn twist along one of the curves $\gamma_1, \gamma_3,\dots, \gamma_{n}$ is a \emph{bridge twist}, and a Dehn twist along one of the curves $\gamma_2,\gamma_4,\dots, \gamma_{n-1}$ is a \emph{river twist}.

\section{Taut Foliations}\label{sec:taut foliations}

In this section we study the existence of taut foliations on the surgeries on fibered hyperbolic two-bridge links, proving the foliation part of
Theorem \ref{thm: main theorem}. Branched surfaces will be our main tool. In Subsection \ref{subsec: background} we introduce them and recall some of their basic properties, together with the main result of \cite{L2}. In Section \ref{subsec: constructing branched surfaces} we recall a general method to construct branched surfaces in fibered manifolds with boundary and in Section \ref{subsec: boundary train tracks} we discuss their boundary train tracks. In Section \ref{subsec: Fibered hyp two-bridge links} we focus our attention on surgeries on fibered hyperbolic two-bridge links: we subdivide them in four families with Lemma \ref{lemma:subdivision in families} and then study each of these families separately in Sections \ref{subsec: family 1}--\ref{subsec: family 4}.

\subsection{Background}\label{subsec: background}

In this and in the next sections we assume familiarity with the basic notions of the theory of train tracks; see \cite{PH} for reference. In the cases of our interest train tracks can also have bigons as complementary regions. 

We now recall some basic facts about branched surfaces. We refer to \cite{FO} and \cite{O} for more details.

\begin{defn}
A \emph{branched surface with boundary} in a $3$-manifold $M$ (possibly with boundary) is a closed subset $B\subset M$ that is locally diffeomorphic to one of the models in $\matR^3$ of Figure \ref{branched surface}a) or to one of the models in the closed half space of Figure \ref{branched surface}b), where $\partial B:= B\cap \partial M$ is represented with a bold line.
\begin{figure}[]
    \centering
    \includegraphics[width=0.8\textwidth]{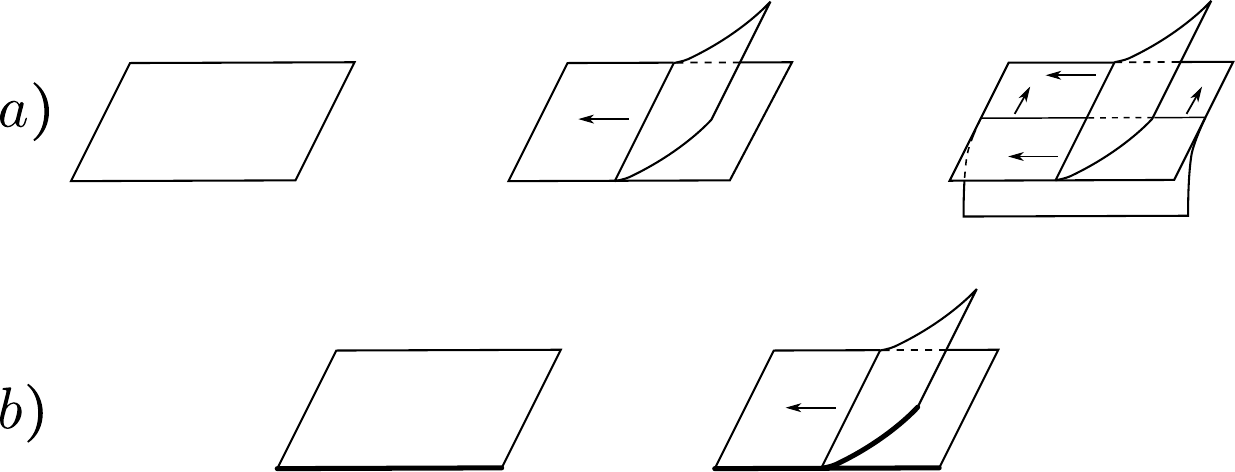}
    \caption{Local models for a branched surface, with cusp directions. The bolded regions lie in $\partial M$.}
    \label{branched surface}
\end{figure}

\end{defn}

Branched surfaces generalise the concept of train tracks from surfaces to $3$-manifolds. When the boundary of $B$ is non-empty it defines a train track $\partial B$ in $\partial M$.

If $B$ is a branched surface it is possible to identify two subsets of $B$: the \emph{branch locus} and the set of \emph{triple points}. The branch locus is defined as the set of points where $B$ is not locally homeomorphic to a surface. It is self-transverse and intersects itself in double points only. The set of triple points of $B$ can be defined as the points where the branch locus is not locally homeomorphic to an arc. For example, the rightmost model of Figure \ref{branched surface}a) contains a triple point.

The complement of the branch locus in $B$ is a union of connected surfaces. The abstract closures of these surfaces under any path metric on $M$ are called the \emph{branch sectors} of $B$. Analogously, the complement of the set of the triple points inside the branch locus is a union of $1$-dimensional connected manifolds. Moreover, to each of these manifolds we can associate an arrow in $B$ pointing in the direction of the smoothing, as in Figure \ref{branched surface}. We call these arrows \emph{cusp directions}.

In \cite{L}, Li introduces the notion of \emph{sink disc}.

\begin{defn}
Let $B$ be a branched surface in $M$ and let $S$ be a branch sector in $B$.
We say that $S$ is a \emph{sink disc} if $S$ is a disc (possibly with $S  \cap \partial M\neq \emptyset$) the branch direction of any smooth curve or arc in $\partial S\setminus \partial M$ points into $S$.
\end{defn}

In Figure \ref{discs}  some examples of sink discs are depicted. The bold lines represent the intersection of the branched surface with $\partial M$. Notice that the intersection $\partial S \cap \partial M$ can also be disconnected.

\begin{figure}[]
    \centering
    \includegraphics[width=1\textwidth]{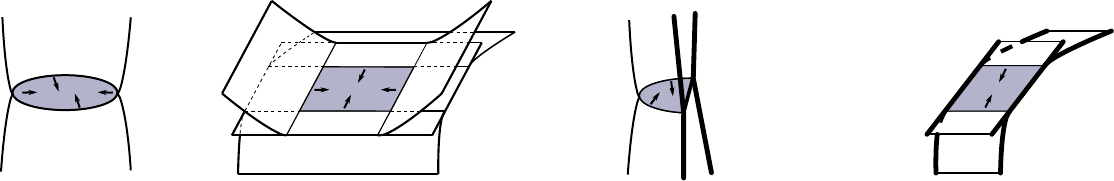}
    \caption{Examples of sink discs.}
    \label{discs}
\end{figure}

In \cite{L}, Li introduces the definition of laminar branched surface and in \cite{L2} he generalises this definition to branched surfaces with boundary as follows. For the definition of trivial bubble we refer the reader to \cite{L,L2}.

\begin{defn}[\cite{L, L2}]\label{laminar}
Let $B$ be a branched surface in a $3$-manifold $M$. We say that $B$ is \emph{laminar} if $B$ has no trivial bubbles and the following hold:

\begin{enumerate}
    \item $\partial_h N_B$ is incompressible and $\partial$-incompressible in $M\setminus {\rm int}(N_B)$, and no component of $\partial_h N_B$ is a sphere or a properly embedded disc in $M$;
    \item there is no monogon in $M\setminus {\rm int}(N_B)$, i.e. no disc $D\subset M\setminus {\rm int}(N_B)$ such that $\partial D=D\cap N_B=\alpha\cup \beta$, where $\alpha$ is in an interval fiber of $\partial_v N_B$ and $\beta$ is an arc in $\partial_h N_B$;
    \item $M\setminus {\rm int}(N_B)$ is irreducible and $\partial M\setminus {\rm int}(N_B)$ is incompressible in $M\setminus {\rm int}(N_B)$;
    \item $B$ contains no Reeb branched surfaces (see \cite{GO} for the definition);
    \item $B$ has no sink discs.
\end{enumerate}
\end{defn}

The key property of laminar branched surfaces is that they fully carry essential laminations\footnote{For the definition of essential laminations see \cite{GO}, but we will not need their properties for our purposes.} \cite{L} and when the ambient manifold $M$ has torus boundary, they can be used to construct essential laminations in fillings of $M$, as the following theorem shows.

\begin{teo}[\cite{L2}]\label{boundary train tracks}
Let $M$ be an irreducible and orientable $3$-manifold whose boundary is union of $k$ incompressible tori $T_1,\dots, T_k$. Suppose that $B$ is a laminar branched surface in $M$ such that $\partial M\setminus \partial B$ is a union of bigons. Then for any multislope $(s_1,\dots, s_k)\in \overline{\matQ}^k$ that is realised by the train track $\partial B$, if $B$ does not carry a torus that bounds a solid torus in $M(s_1,\dots,s_k)$, there exists an essential lamination $\Lambda$ in $M$ fully carried by $B$ that intersects $\partial M$ in parallel simple curves of multislope $(s_1,\dots, s_k)$. Moreover this lamination extends to an essential lamination of the filled manifold $M(s_1,\dots, s_k)$.
\end{teo}

A train track \emph{realises} a slope $r$ if it can be splitted into finitely many copies of that slope, see \cite{PH}. An equivalent, but more computationally useful definition, is given in Subsection 
\ref{subsec: boundary train tracks}.

\begin{rem}
In \cite{L2} the statement of the theorem is given for $M$ with connected boundary but, as already observed in \cite{KR1}, if $M$ has multiple boundary components we can split $B$ in a neighbourhood of each boundary tori $T_i$ and the same proof of \cite{L2} works.
\end{rem}

\begin{rem}
The statement of Theorem \ref{boundary train tracks} is slightly more detailed than the version of \cite{L2}. The details we have added come from the proof of Theorem \ref{boundary train tracks}. In fact the idea of the proof is to split the branched surface $B$ in a neighbourhood of $\partial M$ so that it intersects $T_i$ in parallel simple closed curves of slopes $s_i$, for $i=1,\dots k$. In this way, when gluing the solid tori, we can glue  meridional discs of these tori to $B$ to obtain a branched surface $B(s_1,\dots, s_k)$ in $M(s_1,\dots, s_k)$ that is laminar and that by \cite[Theorem~1]{L} fully carries an essential lamination. In particular, this essential lamination is obtained by gluing the meridional discs of the solid tori to an essential lamination in $M$ that intersects $T_i$ in parallel simple closed curves of slopes $s_i$, for $i=1,\dots, k$.
\end{rem}

\subsection{Constructing branched surfaces in fibered manifolds}\label{subsec: constructing branched surfaces}
In this section we recall a general method to build branched surfaces in mapping tori. This will be the starting point to construct taut foliations on surgeries on fibered two-bridge links.

Let $S$ be a connected oriented surface with boundary, let $h$ be an orientation preserving homeomorphism of $S$ fixing $\partial S$ pointwise and let $M_h$ be the mapping torus associated to $(S,h)$.

We consider pairwise disjoint properly embedded arcs $\alpha_1,\dots, \alpha_k$ in $S$ and discs $\overline{D}_i=\alpha_{i}\times[0,1]\subset S\times [0,1]$. Each of these discs has a ``bottom'' boundary, $\alpha_i\times \{0\}$, and a ``top'' boundary, $\alpha_i\times \{1\}$. When we consider the images of these discs in $M_h$ under the projection map
$$
S\times [0,1]\rightarrow M_h
$$
we have that the bottom and top boundaries become respectively $\cup_{i}\alpha_i\subset S$ and $\cup_{i}h(\alpha_i)\subset~S$.

We perturb the discs $\overline{D}_i$'s in a neighbourhood of $S\times \{1\}\subset S\times [0,1]$ so that when projected to $M_h$ their top boundaries define a family of arcs, that we still denote $h(\alpha_i)$'s, in $S$ such that for each $i,j\in \{1,\dots, k\}$ the intersection between $\alpha_i$ and $h(\alpha_j)$ is transverse and the endpoints of $\alpha_i$ and $h(\alpha_i)$ are disjoint.  We also denote by $D_i$ the projected perturbed disc contained in $M_h$ and we refer to these discs as \emph{product discs.}
If we assign (co)orientations to these discs, since $S$ is (co)oriented, we can smoothen $S\cup D_1 \cup \cdots \cup D_k$ to a branched surface $B$ by imposing that the smoothing preserves the co-orientation of $S$ and of the discs. In particular, each disc has two possible co-orientations and hence it can be smoothed in two different ways. This operation is demonstrated in Figure \ref{smoothings}, where $S$ is a torus with an open disc removed.

\begin{figure}[]
    \centering
    \includegraphics[width=0.9\textwidth]{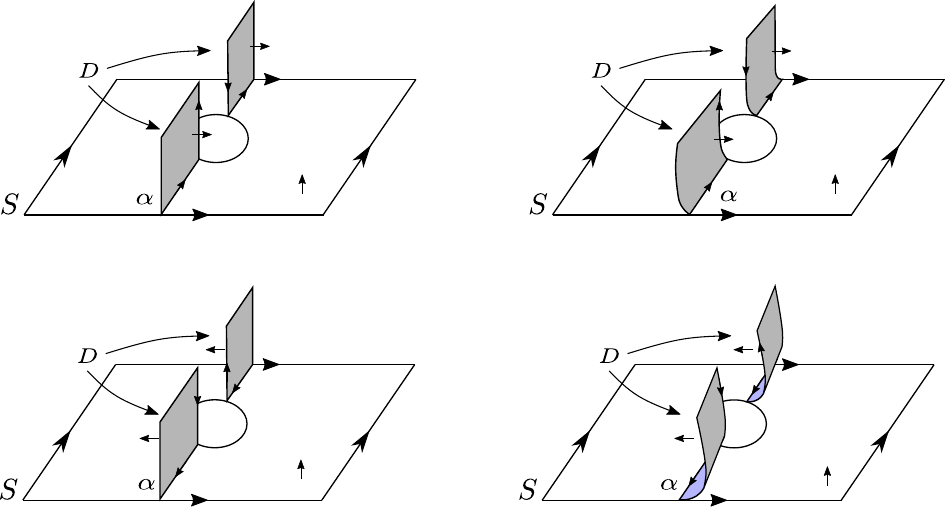}
    \caption{How to smoothen $S\cup D$ according to the co-orientations.}
    \label{smoothings}
\end{figure}

The following proposition shows that in this setting, under very mild hypothesis on the branched surface $B$, we can get taut foliations on fillings on $M_h$. The proof is obtained by combining \cite[Lemma~3.16]{S} and \cite[Lemma~3.23]{S} and uses as a fundamental result Theorem \ref{boundary train tracks}.
\begin{prop}\label{prop: no sink implica taut foliation}
Suppose that $B$ is a branched surface constructed as described above. Suppose also that $B$ has no sink discs and that $S\setminus \cup_{i=1}^k{\alpha_i}$ has no disc components. If $(r_1,\dots,r_n)$ is a multislope realised by $\partial B$ then $M_h(r_1\dots,r_n)$ contains a co-orientable taut foliation. More precisely there exists a co-orientable taut foliation in $M_h$ intersecting the boundary component $T_i$ in a foliation by curves of slopes $r_i$, for $i=1,\dots, n$.\qed
\end{prop}

\subsection{Boundary train tracks}\label{subsec: boundary train tracks}
We now briefly discuss the boundary train tracks of the branched surfaces constructed with the procedure described above. We start by describing an explicit way to compute the slopes realised by a train track. 

Let $\tau$ be an oriented train track in a torus $T$. Fix a meridian-longitude basis $\mu, \lambda$ and use it to identify slopes on $T$ with elements in $\overline{\matQ}$.
It is possible to compute the slopes realised by a train track $\tau$ by endowing it with \emph{weight systems}. A weight system $w$ on $\tau$ is the assignment of a positive rational number, called \emph{weight}, to each sector of $\tau$ so that at each branch point of $\tau$ the sum of the weights of the incoming sectors is equal to the sum of the outgoing one. Given that $\tau$ is oriented, we can associate to such a weight system the rational number $\frac{w_{\mu}}{w_{\lambda}}$, where $w_{\mu}$ and $w_{\lambda}$ are the \emph{weighted} intersections of the train track with our fixed meridians $\mu$ and longitudes $\lambda$, as we would do with oriented simple closed curves. This quotient can be interpreted as a \emph{slope} in $T$ and it can be proved that the slopes $\frac{p}{q}$ obtained in this way are exactly those realised by the train track. For details, see \cite{PH}. Figure \ref{fig:train track example} shows an example of a train track with weight systems.

\begin{figure}[]
    \centering
    \includegraphics[width=0.5\textwidth]{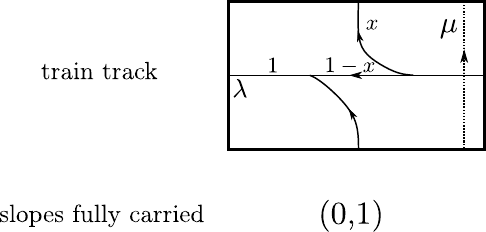}
    \caption{An example of train track $\tau$ with a weight system $w_x$, together with the set of slopes realised by $\tau$. Notice that in order for the weights to be positive, $x$ must be a rational number in $(0,1)$.}
    \label{fig:train track example}
\end{figure}

We now focus our attention on a branched surface $B$ in a mapping torus $M_h$ obtained by adding product discs to the fiber $S$. We start by fixing a meridian $\mu_j$ and a longitude $\lambda_j$ for each boundary component $T_j$ of $M_h$. We take $\lambda_j$ to be $S\cap T_j$, with the orientation induced by $S$, and as meridian the curve $\mu_j=\frac{\{x_j\}\times [0,1]}{\sim_h}$, where $x_j\in S\cap T_j$, oriented in the direction of ascending $t\in [0,1]$. Notice that when $T_j$ is oriented as the boundary of $M_h$, the algebraic intersection $<\mu_j, \lambda_j>$ is negative. By using this meridian-longitude basis we can identify slopes on $T_j$ with elements in $\overline{\matQ}$.

Given a properly embedded oriented arc $\alpha$ in $S$ we denote its endpoints with $\alpha(0)$ and $\alpha(1)$, so that $\alpha$ goes from $\alpha(0)$ to $\alpha(1)$.
Let us now consider a product disc $D$ in the branched surface $B$, spanned by an arc $\alpha$. The fixed orientation on $D$ induces an orientation on $\alpha$ and on the two arcs of intersections $D\cap \partial M_h$. More precisely the arc in $D\cap \partial M_h$ containing $\alpha(1)$ is oriented in the direction of ascending $t\in [0,1]$, and the one containing $\alpha(0)$ is oriented in the direction of descending $t\in [0,1]$ (see Figure \ref{smoothings} for an example). For a fixed boundary component $T$ of $M_h$ the four possible configurations of one arc of $D\cap T$ are described in Figure \ref{fig:train track table}, together with the set of slopes fully carried by the train tracks given by one component of $D\cap T$ and the longitude of $T$.

\begin{figure}[]
    \centering
    \includegraphics[width=1\textwidth]{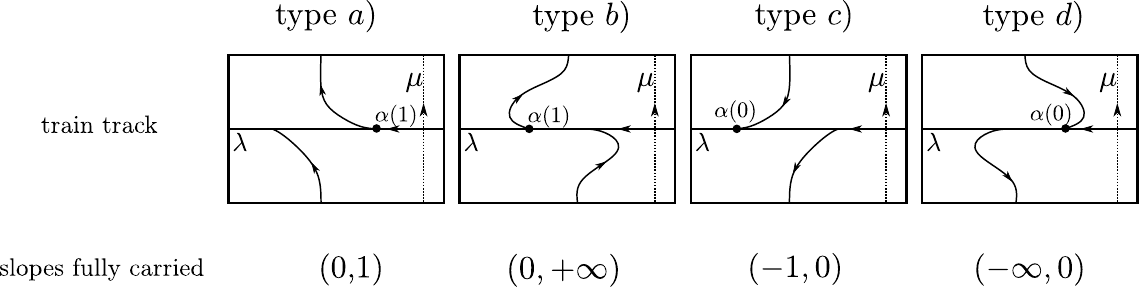}
    \caption{Possible train tracks given by the longitude of $T$ and one arc of intersection in $T\cap D$, together with the intervals of slopes they fully carry.}
    \label{fig:train track table}
\end{figure}

In the course of the paper, we will mostly encounter train tracks of the four types described in Figure \ref{fig:train track table}, or those that we call of \emph{type  $t_1+\cdots + t_m$}, where $t_i$ is $a), b), c)$ or $d)$, for $i=\{1,\dots, m\}$. These are obtained by juxtaposing train tracks of type $t_i$, for $i=\{1,\dots, m\}$, preserving the cyclic order. See Figure \ref{fig:type a+c} for an example.

\begin{figure}[]
    \centering
    \includegraphics[width=0.6\textwidth]{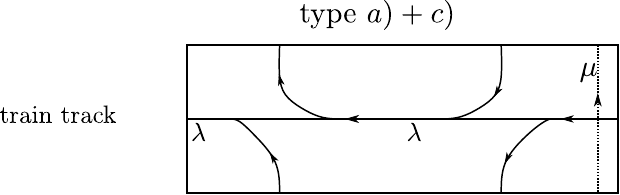}
    \caption{A train track of type $a)+c)$.}
    \label{fig:type a+c}
\end{figure}

The next lemma -- which follows by a direct computation with weight systems -- will be useful when computing the sets of slopes realised by our train tracks.

\begin{lemma}\label{lemma: unlinked}
Let $\tau$ be a train track of type  $t_1+\cdots + t_m$. Then, a slope $r$ is realised by $\tau$ if and only if $r=r_1+ \cdots  + r_m$, where $r_i$ is a  slope realised by the train track of type $t_i$.\qed
\end{lemma}
In many of the cases of our interest, the boundary train tracks will satisfy the hypothesis of Lemma \ref{lemma: unlinked} and so we can use it to compute the slopes that they realise. Two examples of train tracks that do not satisfy the hypothesis of the previous lemma can be found in Figure \ref{fig: Ln branched surface 2}.

\subsection{Fibered hyperbolic two-bridge links}\label{subsec: Fibered hyp two-bridge links}

We are now ready to construct foliations on surgeries on the hyperbolic fibered two-bridge links. The general strategy is simple: we have from Equation \eqref{eq: monodromia} explicit descriptions of the monodromies of these links and we want to construct branched surfaces in the way described in Section \ref{subsec: constructing branched surfaces}. If we are able to construct these branched surfaces so that they have no sink discs, then by Proposition \ref{prop: no sink implica taut foliation} we can deduce that all the surgeries corresponding to the multislopes realised by these branched surfaces contain co-orientable taut foliations. For this reason, we will have to study which multislopes are realised by the boundary train tracks.

Before we begin, we make the following remark and establish some conventions.

\begin{rem}
In \cite{R1} it is proved that all nontrivial surgeries on fibered hyperbolic knots with fractional Dehn twist coefficient zero contain a co-orientable taut foliation. A similar result does not hold for links. In fact, all fibered hyperbolic two-bridge links have fractional Dehn twist coefficient zero (on both boundary components). This can be proved by finding appropriate arcs moved to the left and to the right and by using \cite[Proposition~3.1]{HKM1}.  Nonetheless, this family contains infinitely many $L$-space links.
\end{rem}
\textbf{Conventions.} Thorough the section we will use the following conventions:
\begin{itemize}
\item we fix a fiber surface $S$ for the two-bridge link $L=L(2b_1,\dots, 2b_n)$ and we fix its orientation as in Figure \ref{fig:plumbing two bridge}. With the induced orientation, $L$ has linking number 
$$
\lk(L)=\sum_{i=0}^{k}b_{2i+1}
$$
where $n=2k+1$;
\item when a link is fibered there is a natural choice of meridians and longitudes that is in general different from the one induced by the ambient manifold $S^3$, obtained as follows. We identify $S^3\setminus {\rm int}(N_L)\cong \frac{S\times [0,1]}{\sim_h}$, where $N_L$ is a tubular neighbourhood of $L$. We fix a point $x_i$ in each boundary component of $S$ and we consider the curves $\mu_i=\frac{\{x_i\}\times [0,1]}{\sim_{h}}$ oriented in the direction of ascending $t\in [0,1]$ as meridians and the boundary components $\lambda_i$ of $S$ as longitudes. By definition of fibered link, the meridians defined in this way coincide with the usual meridians of the link. On the other hand these longitudes do not coincide in general with the canonical longitudes of the link components. In fact, letting $l_i$ denote the canonical longitude of $K_i$, we have that

\begin{equation*}\label{eq: cambio coefficienti}
    \lambda_i+\sum_{j\ne i}\lk(K_i, K_j)\mu_i=l_i 
    \tag{$\star$}
\end{equation*}
as elements in $H_1(\partial N_{K_i}, \matZ)$, where $N_{K_i}$ is the connected component of $N_L$ containing $K_i$.

From now on we will refer to the bases $(\mu_i, \lambda_i)$ as the \emph{Seifert framing}, and to the bases $(\mu_i,l_i)$ as the \emph{canonical framing}. Notice that it follows from Equation \eqref{eq: cambio coefficienti} that the longitudes given by the Seifert framing do not depend on the choice of the Seifert surface.
Unless otherwise stated we use Seifert framings;
\item we will always suppose $n>1$, because when $n=1$ the only links obtained in this way are the Hopf links and we are interested in hyperbolic links;
\item we construct branched surfaces by considering \emph{oriented} arcs in the fiber surface $S$ and then by attaching product discs as in Section \ref{subsec: constructing branched surfaces}. We will always co-orient the discs in the following way: we orient them so that the orientations on their boundaries induce the given orientation on the arcs and then we use the orientation of the ambient manifold to co-orient them. Analogously, the co-orientation of the fiber $S$ is obtained by using the orientation of $S$ and of the ambient manifold.
A good way to keep in mind the cusps directions of branched surfaces constructed in the way is the following: looking at the positive side of $S$, the cusps directions is pointing right along the arcs $\alpha$ and $\beta$ with respect to their orientations and is pointing left along the oriented arcs $h(\alpha)$ and $h(\beta)$ with respect to their orientations. See Figure \ref{fig: river twist};
\item drawings will usually show the positive side of the (abstract) fiber surface $S$;
\item we will usually omit $1$-handles from the drawings, and we draw the attaching arcs of the $1$-handles in pink. Compare Figure \ref{fig:curves} with Figure \ref{fig: bridge twist pos};
\item the arcs $\alpha, \beta, \dots$ lie on the positive side of $S$ and are depicted with solid lines; their images $h(\alpha), h(\beta), \dots$ via the monodromy $h$ lie on the negative side of $S$ and are depicted with dashed lines;
\item we will usually denote the sectors of a branched surface in $S$ with letters $\mathcal{A}, \mathcal{B}, \dots$ .
\end{itemize}

With the following lemma we partition fibered hyperbolic two-bridge links in few families and study all of them separately.

\begin{lemma}\label{lemma:subdivision in families}
Let $L=L(2b_1,\dots, 2b_n)$ with $|b_i|=1$ for all $i$'s and $n$ odd. If $L$ is not a torus link then, up to mirroring $L$, $(b_1,\dots, b_n)$ satisfies:
\begin{itemize}
\item \emph{Family 1:} there exist indices $l$ and $m$ such that $b_{2l}\ne b_{2m}$;
    \item \emph{Family 2:} $(b_1,\dots, b_n)=(b_1,-1, b_3, \dots, -1, b_n)$ and there exist indices $j\ne l\ne m$ such that $b_{2j+1}=b_{2m+1}=-1$ and $b_{2l+1}=1$;
    \item \emph{Family 3:} $(b_1,\dots, b_n)=(-1, -1, -1, \dots, -1, -1)$;
    \item \emph{Family 4:} $(b_1,\dots, b_n)=(b_1,-1, b_3, \dots, -1, b_n)$ where exactly one $b_{2j+1}$ is $-1$. 
\end{itemize}
\end{lemma}
\begin{proof}
Suppose that $L$ does not belong to Family $1$. Then, up to mirroring, we can suppose that $b_{2l}=-1$ for all indices $l$. We now consider the integers $b_{2j+1}$'s:
\begin{itemize}
\item if there are at least two of them that are negative and one that is positive, then $(b_1,\dots, b_n)$ belongs to Family $2$;
\item if none of them is positive, then $(b_1,\dots, b_n)$ belongs to Family $3$;
\item if exactly one of them is negative, then $(b_1,\dots, b_n)$ belongs to Family $4$;
\item if none of them is negative, then $L=(2, -2, 2, \dots, -2, 2)$ and is a torus link.
\end{itemize}
This concludes the proof.
\end{proof}

Before beginning the study of these families, we briefly summarise what we will prove and the general strategy of the proofs. For links in families $1, 2$ and $3$ we will construct taut foliations on all rational surgeries. The links in family $4$ will be divided into two subfamilies, the first containing the links with $b_{2j+1}=-1$ for some $j$ such that $2j+1\ne 1, n$ and the second containing those with $b_1=-1$ or $b_n=-1$. The links in the first subfamily will have taut foliations on all rational surgeries, while those in the second will be isotopic to the links $L_n$ in Figure \ref{figure:link L_n} and we will construct taut foliations on all surgeries in $\big{(}(-\infty, n)\times \matQ\big{)}\cup \big{(}\matQ\times (-\infty, n)\big{)}$.

For families $1$ and $2$, the taut foliations will be constructed by using branched surfaces obtained by adding product discs to the fiber surface. For families $3$ and $4$ we need a more elaborate strategy, and we will drill out one or two carefully chosen loops in the link complements to get new fibered $3$-component and $4$-component links. We will then study these new links, and construct branched surfaces in their exterior by adding product discs to their fiber surfaces.

\subsection{Study of the links in Family 1}\label{subsec: family 1}

Recall from the end of Section \ref{sec: two-bridge background} the definition of river twist and bridge twist.

\begin{lemma}\label{lemma: river twist}
Let $L=L(2b_1,\dots, 2b_n)$ with $n$ odd and $|b_i|=1$ for all $i$'s, and let $h$ denote its monodromy as in Equation \eqref{eq: monodromia}. Let $M$ denote the exterior of $L$. Then
\begin{enumerate}
    \item if there is at least one positive (resp. negative) river twist in the factorisation of $h$, the manifold $M(r_1,r_2)$ contains a co-orientable taut foliation for every multislope $(r_1,r_2)\in (-\infty, 1)^2$ (resp. for all $(r_1,r_2)\in (-1, +\infty)^2$); see Figure \ref{fig: river twist region}a)-b);
    \item if there are two river twists with different exponents in the factorisation of $h$, the manifold $M(r_1,r_2)$ contains a co-orientable taut foliation for every multislope $(r_1,r_2)\in \big{(}(-1, +\infty)\times (-\infty, 1)\big{)}\cup \big{(}(-\infty, 1)\times(-1 +\infty)\big{)}$; see Figure \ref{fig: river twist region}c).
\end{enumerate}
\end{lemma}
  
\begin{figure}[]
    \centering
    \includegraphics[width=0.7\textwidth]{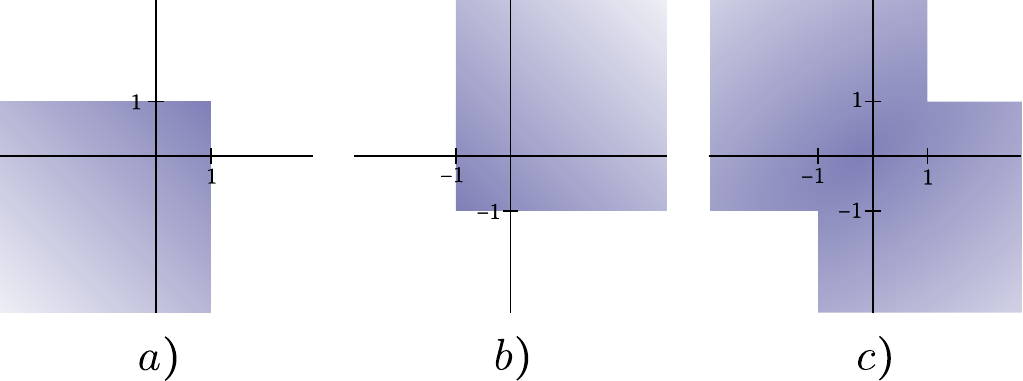}
    \caption{From left to right, the slopes $(r_1, r_2)$ in the coloured region yield manifolds with co-orientable taut foliations in the case where there is respectively: at least one positive river twist, at least one negative river twist, two river twists with different exponents in the factorisation of the monodromy $h$.}
    \label{fig: river twist region}
\end{figure}

\begin{proof}\begin{enumerate}[leftmargin=*]
  \item
Suppose that there is a positive river twist along the curve $\gamma_i$. We consider the arcs $\alpha$ and $\beta$ as in Figure \ref{fig: river twist}. The oriented arcs $\alpha$ and $\beta$ determine a co-oriented branched surface $B$ obtained by attaching two discs to the fiber surface $S$ as described in Section \ref{subsec: constructing branched surfaces}. Since $n>1$, $S$ is not an annulus and therefore the complement of $\alpha \cup \beta$ in $S$ has no disc components. Due to the fact that we have chosen $\alpha$ and $\beta$ so that they are disjoint from $\gamma_j$ for $j\ne i$ it follows that $h(\alpha)=\tau_i(\alpha)$ and $h(\beta)=\tau_i(\beta)$, as depicted in Figure \ref{fig: river twist}. In Figure \ref{fig: river twist} we have also labelled the branch locus of $B$ with the cusps directions and denoted with capital letters $\mathcal{A}, \mathcal{B}, \mathcal{C}, \mathcal{D}$ the four sectors of the branched surface $B$ in $S$ and none of them is a sink disc. For this reason we can apply Proposition \ref{prop: no sink implica taut foliation} and deduce that $M(r_1,r_2)$ supports a co-orientable taut foliation for all the multislopes $(r_1, r_2)$ realised by $\partial B$.
On each boundary component of $M$ we obtain a train track, depicted in Figure \ref{fig: train track river twist}$a)$, of type $a)+d)$. Therefore by using Lemma \ref{lemma: unlinked} we have that the boundary train tracks of $B$ realise all the multislopes in $(-\infty, 1)^2$ and by applying Proposition \ref{prop: no sink implica taut foliation} we obtain taut foliations on $M(r_1,r_2)$ for all $(r_1,r_2)\in (-\infty, 1)^2$.

\begin{figure}[]
    \centering
    \includegraphics[width=1\textwidth]{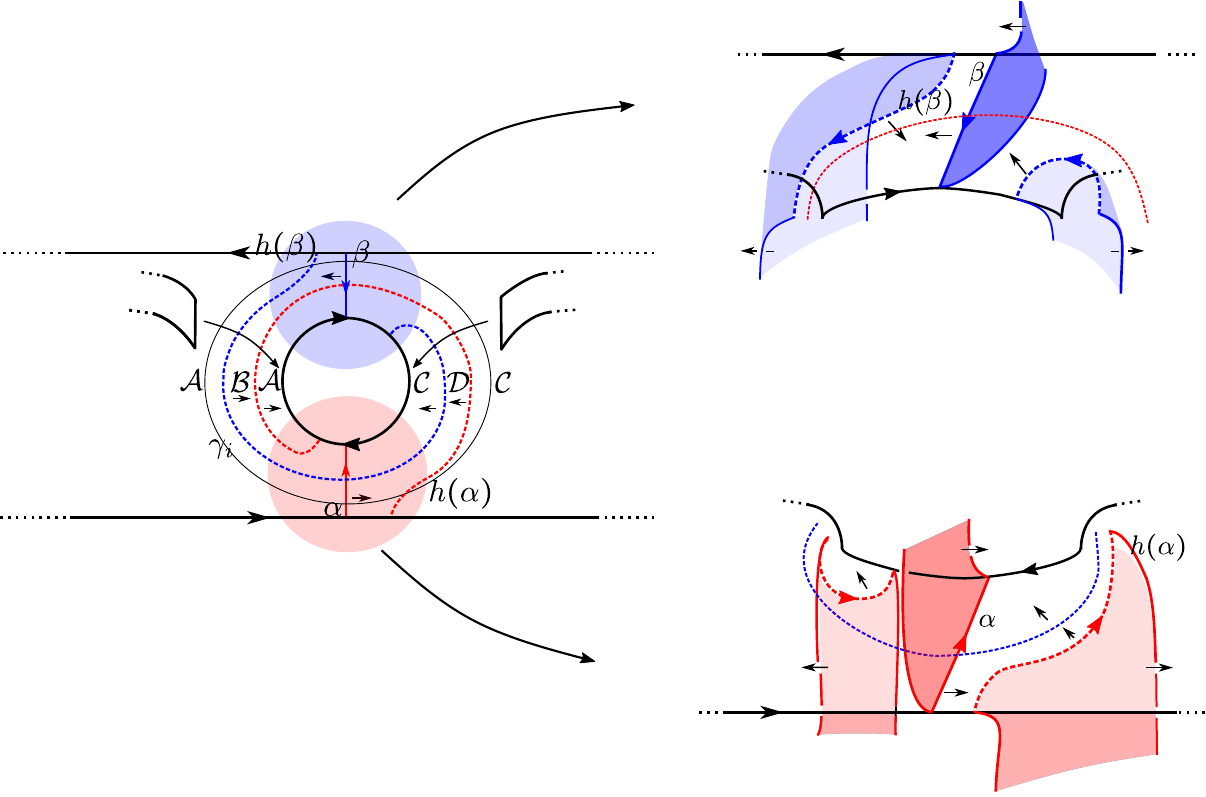}
    \caption{The arcs $\alpha$ and $\beta$ and the co-oriented discs spanned by them. The letters $\mathcal{A}, \mathcal{B}, \mathcal{C}, \mathcal{D}$ denote the sectors of $B$ in $S$.}
    \label{fig: river twist}
\end{figure}

If there is a negative river twist, we consider the same oriented arcs $\alpha$ and $\beta$ (see Figure \ref{fig: river_twist_neg}), and on each of the two boundary components we obtain the train track depicted in Figure \ref{fig: train track river twist}$b)$, that is of type $b)+c)$. This train track realises all the slopes in $(-1, +\infty)$ and so we obtain taut foliations on $M(r_1,r_2)$ for all $(r_1,r_2)\in (-1, +\infty)^2$.

\begin{figure}[]
    \centering
    \includegraphics[width=0.5\textwidth]{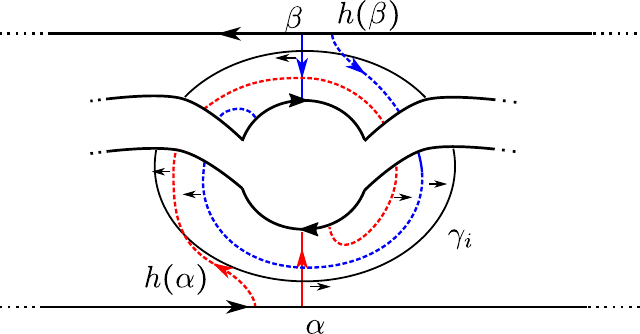}
    \caption{The arcs $\alpha$ and $\beta$ when the river twist is negative.}
    \label{fig: river_twist_neg}
\end{figure}

\begin{figure}[]
    \centering
    \includegraphics[width=0.6\textwidth]{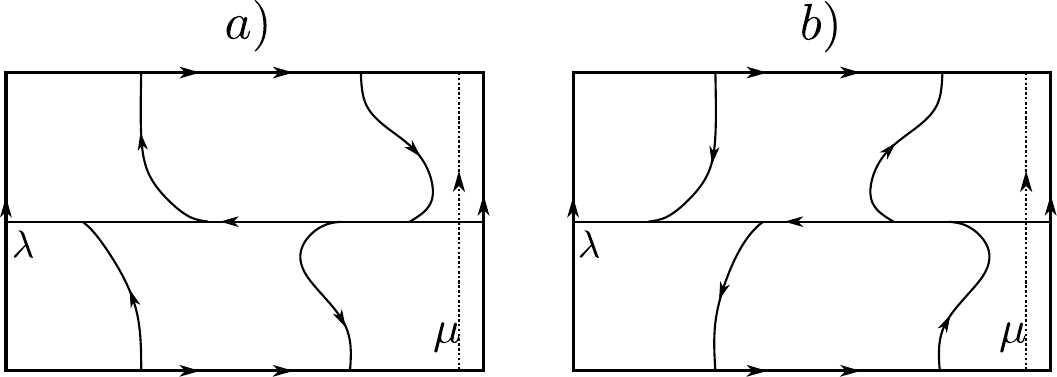}
    \caption{The boundary train tracks in the case where $a)$ there is a positive river twist and $b)$ there is a negative river twist.}
    \label{fig: train track river twist}
\end{figure}

\item Suppose now that there are two river twists with different exponents in the factorisation of $h$ and suppose that the positive one is along the curve $\gamma_i$ and the negative one is along $\gamma_j$. We suppose $i<j$ but the proof does not change if $j<i$. We choose now $\alpha$ and $\beta$ as in Figure \ref{fig: discord_river_twist} and as before we have $h(\alpha)=\tau_i(\alpha)$ and $h(\beta)=\tau_j(\beta)$. Also in this case the complement of $\alpha \cup \beta$ contains no disc components. Moreover the complement of $\alpha\cup \beta\cup h(\alpha) \cup h(\beta)$ in $S$ is connected and this implies that there are no sink discs in the branched surface associated to $\alpha$ and $\beta$.

The boundary train tracks are shown in Figure \ref{fig: discord_river_twist}. The one on the boundary component of $M$ containing the boundary component of $S$ labelled with $1$ is of type $b)+c)$ and by Lemma \ref{lemma: unlinked} it realises all slopes in $(-1, +\infty)$. 
On the other boundary component we get a train track of type $a)+d)$ that hence realises all slopes in  $(-\infty, 1)$. Therefore as a consequence of Proposition \ref{prop: no sink implica taut foliation} we have taut foliations in $M(r_1,r_2)$ for all $(r_1,r_2)\in (-1, +\infty)\times(-\infty, 1)$. As by Lemma \ref{lemma: symmetric} two-bridge links are symmetric, we deduce that there are taut foliations also on the surgeries associated to coefficients $(r_1,r_2)\in
(-\infty, 1)\times(-1, +\infty)$.

\begin{figure}[]
    \centering
    \includegraphics[width=0.9\textwidth]{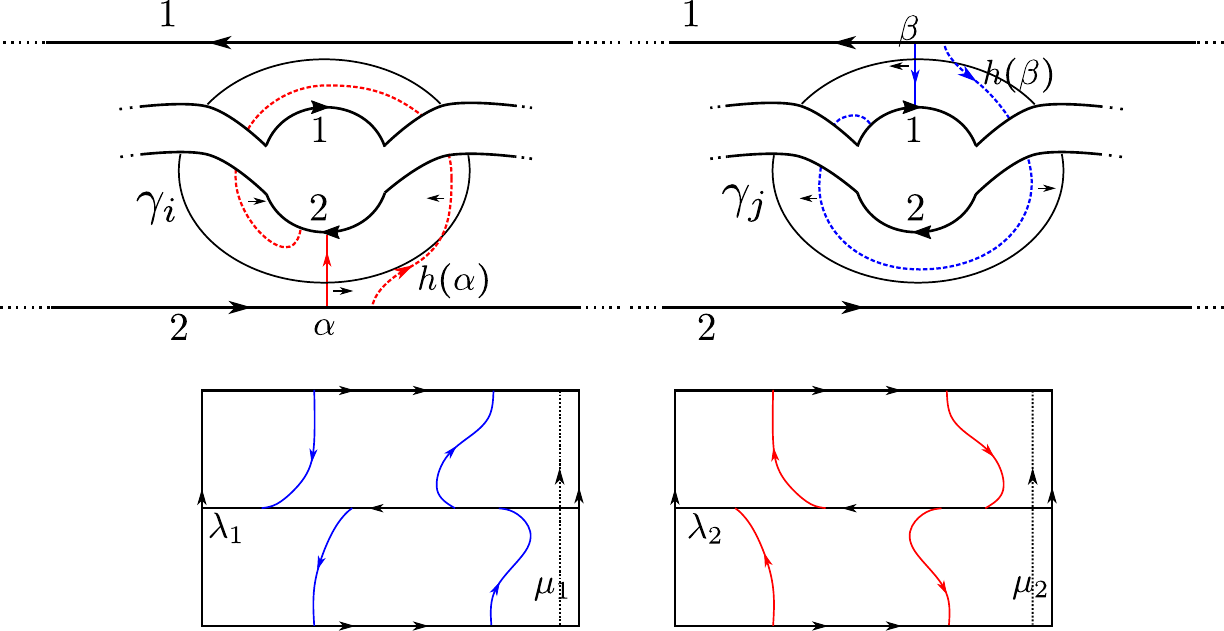}
    \caption{This picture describes the choice of the arcs $\alpha$ and $\beta$ when the twist along the curve $\gamma_i$ is positive and the one along $\gamma_j$ is negative, together with the boundary train tracks of the associated branched surface.}
    \label{fig: discord_river_twist}
\end{figure}

\end{enumerate}

This concludes the proof.
\end{proof}

\begin{rem}
Recall that we are working with Seifert framings. However we have already noticed that the meridians of the Seifert framing coincide with the canonical meridians of $L$. This implies that a surgery coefficient $(r_1, r_2)$ on $L$ is rational with respect to the Seifert framing if and only if it is rational with respect to the canonical framing.
\end{rem}

\begin{cor}\label{cor: prima riduzione}
If the factorisation of the monodromy $h$ has two river twists with different exponents, i.e. if $L$ belongs to Family $1$, then all the rational surgeries on the link $L$ contain co-orientable taut foliations. 
\end{cor}
\begin{proof}
It follows from the first part of Lemma \ref{lemma: river twist} that there are co-orientable taut foliations on $M(r_1,r_2)$ for $(r_1,r_2)\in (-\infty,1)^2\cup (-1, +\infty)^2$ and it follows from the second part of Lemma \ref{lemma: river twist}  that there are co-orientable taut foliations on $M(r_1,r_2)$ for $(r_1,r_2)\in \big{(}(-1, +\infty)\times (-\infty,1)\big{)} \cup \big{(}(-\infty,1)\times(-1, +\infty)\big{)}$. The union of these sets is exactly the set of all rational multislopes.
\end{proof}

\subsection{Study of the links in Family 2}\label{subsec: family 2}
As a consequence of Corollary \ref{cor: prima riduzione}, by taking mirrors if necessary, we can reduce our study to the case where the river twists are all positives, i.e. to links of the form $L=L(2b_1,-2, 2b_3, \dots, -2,2b_n)$ with $n$ odd.

\begin{lemma}\label{lemma: bridge twists}
Let $L=L(2b_1,\dots, 2b_n)$ with $n$ odd and $|b_i|=1$ for all $i$'s, and let $h$ denote its monodromy as in Equation \eqref{eq: monodromia}. Let $M$ denote the exterior of $L$. Then
\begin{enumerate}
    \item if there are at least two positive (resp. negative) bridge twists in the factorisation of $h$, the manifold $M(r_1,r_2)$ contains a co-orientable taut foliation for every multislope $(r_1,r_2)\in (-\infty, 1)^2$ (resp. for all $(r_1,r_2)\in (-1, +\infty)^2$); see Figure \ref{fig: bridge twist region}a)-b);
    \item if there are two bridge twists with different exponents in the factorisation of $h$, the manifold $M(r_1,r_2)$ contains a co-orientable taut foliation for every multislope $(r_1,r_2)\in \big{(}(0, +\infty)\times (-\infty,0)\big{)}\cup \big{(}(-\infty,0)\times(0,+\infty)\big{)}$, see Figure \ref{fig: bridge twist region}c).
\end{enumerate}
\end{lemma}

\begin{figure}[]
    \centering
    \includegraphics[width=0.7\textwidth]{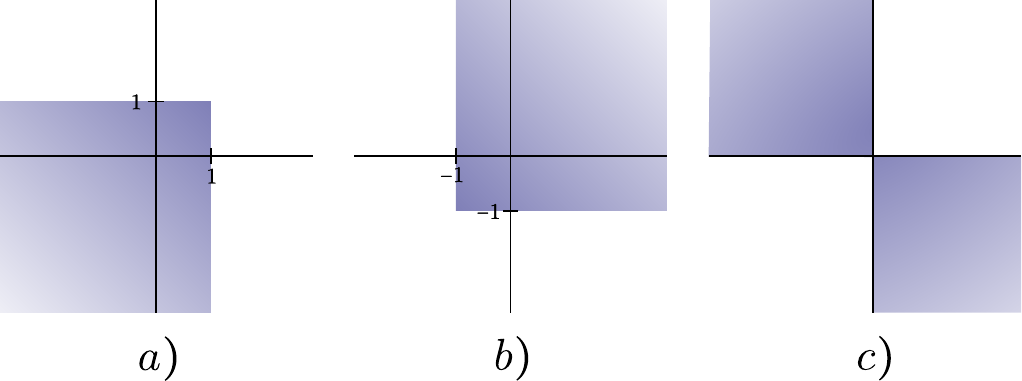}
    \caption{From left to right, the slopes $(r_1, r_2)$ in the coloured region yield manifolds with co-orientable taut foliations in the case where there are respectively: at least two positive bridge twists, at least two negative bridge twists, two bridge twists with different exponents in the factorisation of the monodromy $h$.}
    \label{fig: bridge twist region}
\end{figure}

\begin{proof}
\begin{enumerate}[leftmargin=*]
    \item Suppose that the positive bridge twists are along the curves $\gamma_i$ and $\gamma_j$. We consider the oriented arc $\alpha$ and $\beta$ as in Figure \ref{fig: bridge twist pos}. We have $h(\alpha)=\tau_i(\alpha)$: in fact 
    $$
    h=\underbrace{\tau_2^{\varepsilon_2}\tau_4^{\varepsilon_4}\dots\tau_{2k}^{\varepsilon_{2k}}}_{\text{river twists}}\underbrace{\tau_1^{\varepsilon_1}\tau_3^{\varepsilon_3}\dots\tau_{2k+1}^{\varepsilon_{2k+1}}}_{\text{bridge twists}}
    $$
    and the only bridge twist that has effect on $\alpha$ is $\tau_i$ and the river twists have no effect on $\tau_i(\alpha)$. The same reasoning proves that $h(\beta)=\tau_j(\beta)$. Also in this case we obtain a branched surface that satisfies the hypotheses of Proposition \ref{prop: no sink implica taut foliation}. In fact this branched surface has at most two sectors\footnote{there are two sectors when $i=1$ and $j=n=3$ and there is only one sector otherwise.} $\mathcal{A}$ and $\mathcal{B}$ in $S$ and each of them has some cusp direction on its boundary pointing out of it, as Figure \ref{fig: bridge twist pos} shows. Therefore we just need to study the multislopes realised by the boundary train tracks of $B$. Both of these are of type $a)+d)$ and hence the multislopes realised are the ones in $(-\infty, 1)^2$.
    
\begin{figure}[]
    \centering
    \includegraphics[width=0.65\textwidth]{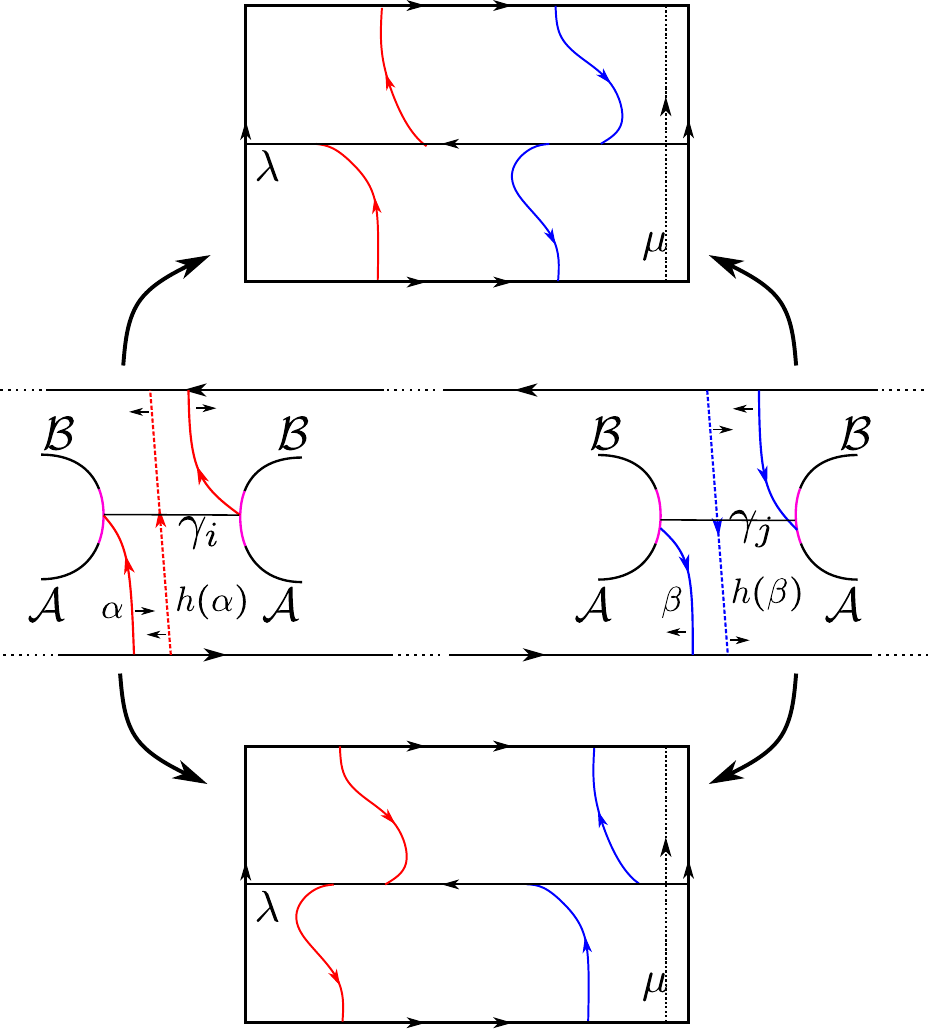}
    \caption{The arcs $\alpha$ and $\beta$, together with their image via the monodromy $h$ and the cusps directions, are depicted. We also describe the train tracks obtained on the boundary of $M$. To simplify the picture we do not draw the $1$-handles; we understand that the pink-coloured lines are pairwise identified in the obvious way.}
    \label{fig: bridge twist pos}
\end{figure}

The case where we have two negative bridge twists is analogous: we choose $\alpha$ and $\beta$ in the same way but now so that they turn right when they meet the curves $\gamma_i$ and $\gamma_j$. Everything works in the same way but now the multislopes realised by the boundary train tracks are the ones in $(-1, +\infty)^2$.
    \item Suppose that are two bridge twists with different exponents in the factorisation of $h$ and suppose that the positive one is along the curve $\gamma_i$ and the negative one is along $\gamma_j$. We choose $\alpha$ and $\beta$ as in Figure \ref{fig: bridge twist discordi}. Also in this case there are at most two sectors $\mathcal{A}$ and $\mathcal{B}$ of the resulting branched $B$ surface in $S$ and none of them is a sink disc. Moreover, the boundary train tracks of $B$ realise all the slopes in $(0,+\infty)$ and $(-\infty, 0)$. In fact on one boundary component we have a train track of type $a)+b)$ and on the other a train track of type $c)+d)$.
\end{enumerate}
    Using the fact that two-bridge links are symmetric, we obtain the statement.
\end{proof}
\begin{figure}[]
    \centering
    \includegraphics[width=0.6\textwidth]{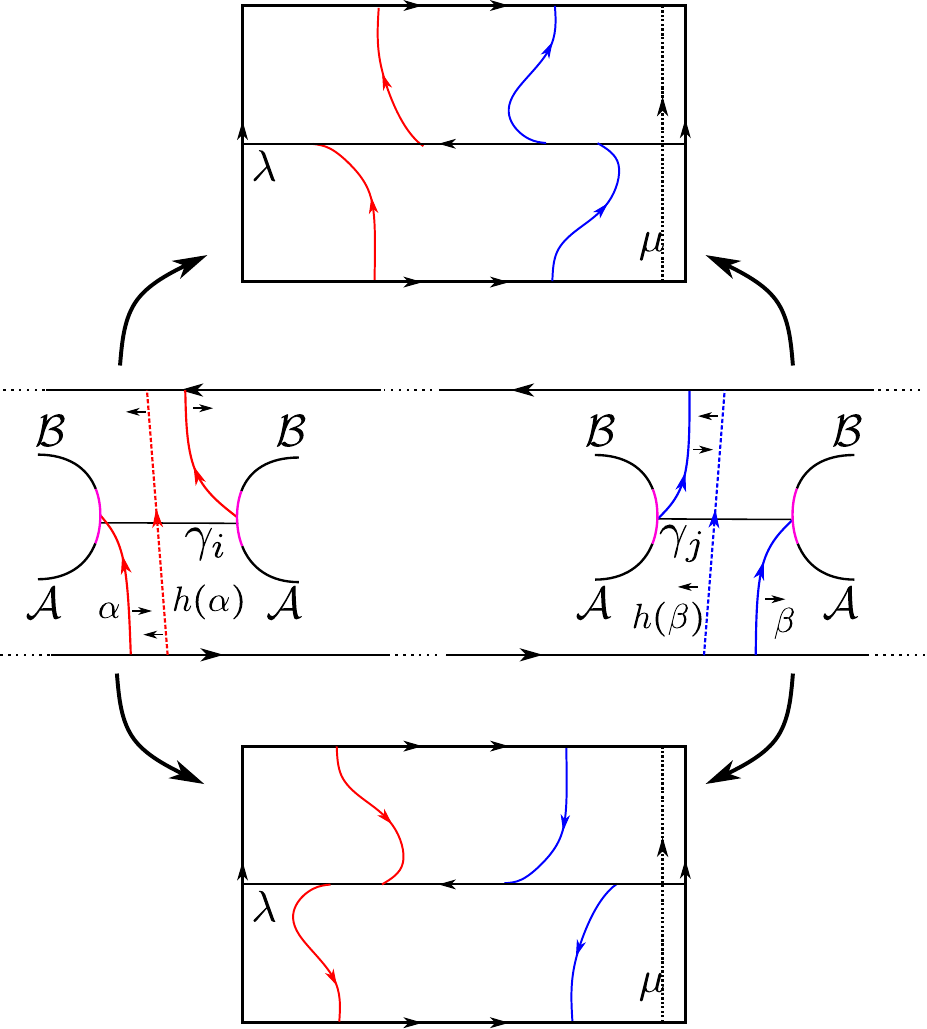}
    \caption{The arcs $\alpha$ and $\beta$ in the case where there are two bridge twist with different exponents and the boundary train tracks realised on $\partial M$.}
    \label{fig: bridge twist discordi}
\end{figure}

\begin{cor}\label{cor: seconda riduzione}
Let $L=L(2b_1,-2,2b_3,\dots, -2,2b_n)$ with $n$ odd and $|b_i|=1$ for all $i$'s, and let $h$ denote its monodromy as in Equation \eqref{eq: monodromia}. If there are at least two negative bridge twists and one positive bridge twist in the factorisation of $h$, i.e. if $L$ belongs to Family $2$, then all the rational surgeries on the link $L$ contain co-orientable taut foliations.
\end{cor}

\begin{proof}
As a consequence of the fact that the factorisation of $h$ contains positive river twists, by Lemma \ref{lemma: river twist} we know that $M(r_1,r_2)$ contains a taut foliation for all the multislopes $(r_1,r_2)\in (-\infty, 1)^2$. Moreover, since there are two negative bridge twists it follows from the first part of Lemma \ref{lemma: bridge twists} that $M(r_1,r_2)$ contains a taut foliation for all the multislopes $(r_1,r_2)\in (-1, +\infty)^2$. As there is also at least one positive bridge twist we can apply the second part of Lemma \ref{lemma: bridge twists} and deduce that $M(r_1,r_2)$ contains a taut foliation for all the multislopes $(r_1,r_2)\in \big{(}(0, +\infty)\times (-\infty,0)\big{)}\cup \big{(}(-\infty,0)\times(0,+\infty)\big{)}$. The union of these sets is exactly the set of all rational multislopes.
\end{proof}
\subsection{Study of the links in Family 3}\label{subsec: family 3}
We now focus our attention on the links composing \emph{Family 3}.
\begin{prop}
Let $L$ be a two-bridge link of the form $L=L(-2, -2, -2, \dots, -2, -2)$, i.e. belonging to Family $3$. Then all the rational Dehn surgeries on $L$ support a co-orientable taut foliation.
\end{prop}
\begin{proof}
It follows by Lemmas \ref{lemma: river twist} and \ref{lemma: bridge twists} that, as the monodromy of $L$ has (at least) two negative bridge twists and (at least) one positive river twist, then all the surgery coefficients contained in $(-\infty, 1)^2\cup (-1, +\infty)^2$ yield manifolds with co-orientable taut foliations. We recall that these coefficients are associated to the Seifert framing.
We now consider two cases:
\begin{itemize}[leftmargin=*]
\item \emph{$L$ is not the link $L(-2, -2 ,-2)$:} we construct a branched surface $B$ whose boundary train tracks realise all the multislopes in $(-\infty, 1)\times (0, +\infty)$. Two-bridge links are symmetric, hence this will imply the statement.
This branched surface is constructed by considering the arcs $\alpha$ and $\beta$ in Figure \ref{fig: family 1 general} and satisfies the hypotheses of Proposition \ref{prop: no sink implica taut foliation}.  In fact the complement of $\alpha \cup \beta$ in $S$ is not a disc, and there is only one sector of $B$ in $S$. Hence $B$ has no sink discs. We can therefore use $B$ to construct foliations on all the surgeries associated to the multislopes realised by its boundary train tracks: one of these is of type $b)$, and so realises all slopes in $(0,+\infty)$, and the other is of type $a)+c)+d)$ and hence realises all slopes in $(-\infty, 1)$.
\begin{figure}[]
    \centering
    \includegraphics[width=0.7\textwidth]{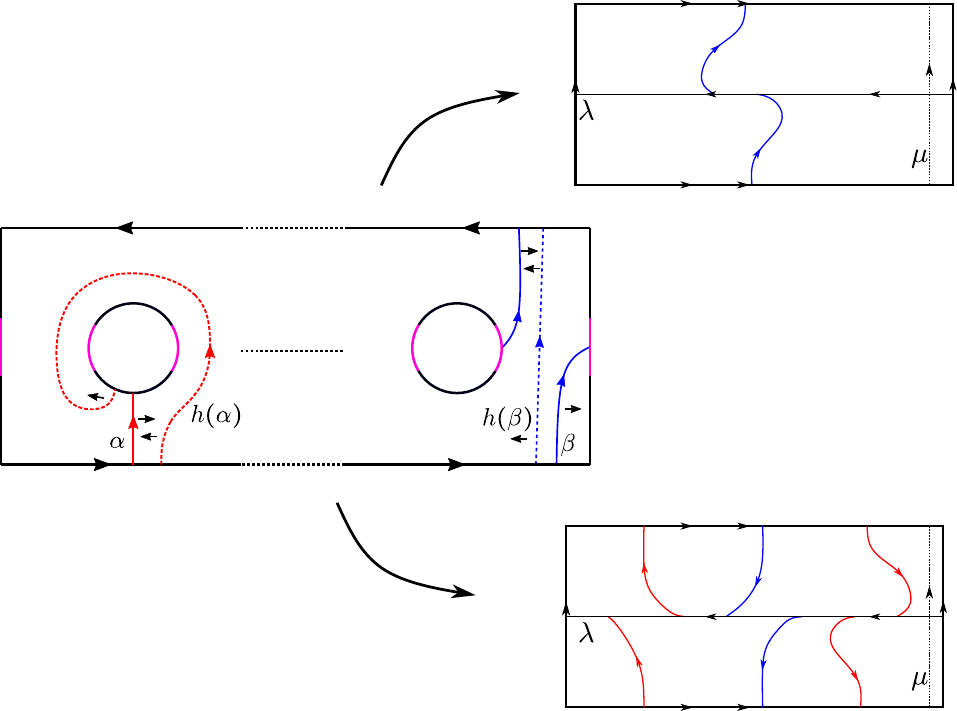}
    \caption{The arcs $\alpha$ and $\beta$ that we consider when $L$ is not the link $L(-2, -2, -2)$, and the boundary train track of the associated branched surface.}
    \label{fig: family 1 general}
\end{figure}

\item \emph{$L=L(-2, -2, -2)$:} to study this case we use an idea that will be useful also later on. We construct taut foliations on all the $(r,s)-$surgeries on $L$, where $r<0$ or $s<0$. This is enough because we already know from Lemma \ref{lemma: bridge twists} that the surgeries associated to $(r,s)\in (-1,+\infty)^2$ contain taut foliations. Observe that $L$ can be described as surgery on a $3$-components link $\mathcal{L}$, as in Figure \ref{fig: family 1 plumbing}. The link $\mathcal{L}$ is also fibered, because it is boundary of a surface obtained via a sequence of Hopf plumbing, as described in Figure \ref{fig: family 1 plumbing}.

\begin{figure}[]
    \centering
    \includegraphics[width=0.8\textwidth]{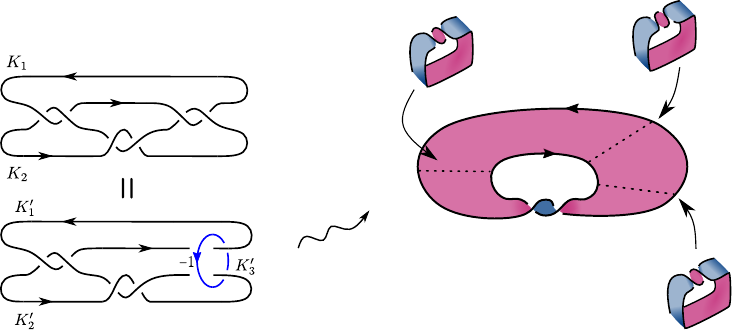}
    \caption{How to obtain the link $L(-2,-2,-2)$ as surgery on a $3$-component link $\mathcal{L}$. We also describe a fiber surface for $\mathcal{L}$, obtained via a sequence of Hopf plumbings.}
    \label{fig: family 1 plumbing}
\end{figure}

Moreover the monodromy of the link $\mathcal{L}$ is given by $h=\tau_4\tau^{-1}_3\tau_2\tau_1^{-1}$, where $\tau_i$ denotes the positive Dehn twist along the curve $c_i$ shown in Figure \ref{fig: family 1 twists}.

\begin{figure}[]
    \centering
    \includegraphics[width=0.4\textwidth]{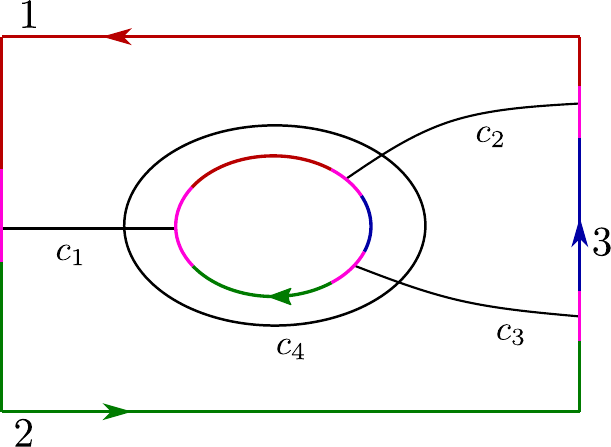}
    \caption{An abstract drawing of the fiber surface for the link $\mathcal{L}$, together with the curves $c_i$'s. We have also coloured the three boundary components of $S$, the boundary component labelled $i$ corresponding to the component $K'_i$ of the link, for $i=1, 2, 3$.}
    \label{fig: family 1 twists}
\end{figure}

This description of $L$ will help us to construct the desired taut foliations. The idea is to find a branched surface in the exterior of $\mathcal{L}$ so that the boundary train tracks realise slope $-1$ on the boundary component associated to $K_3'$. To do this is important to pay attention to how the surgery coefficients change when passing from $\mathcal{L}$ to $L$. Recall that the coefficients of the slopes are written by using the identification given by the Seifert framing. The $(a, b, -1)$-surgery on $\mathcal{L}$ coincides with the $(a-1, b+1)$-surgery on $L$, as the following diagram suggests: 

\begingroup
\mathsurround=0pt
\begin{center}
\begin{tikzcd}
{\overbrace{(a, b, -1)}^{\text{Seifert framing for $\mathcal{L}$}}} \arrow[r, dashed] \arrow[d] & {\overbrace{(a-1, b+1)}^{\text{Seifert framing for $L$}}}              \\
{\underbrace{(a, b+2, -1)}_{\text{Canonical framing for $\mathcal{L}$}}} \arrow[r]               & {\underbrace{(a+1, b+3)}_{\text{Canonical framing for $L$}}} \arrow[u].
\end{tikzcd}
\end{center}
\endgroup

The changes of coefficients indicated by the vertical arrows are a consequence of formula \eqref{eq: cambio coefficienti} and the fact that
\[
\lk(K_1, K_2)=-2, \quad \lk(K_1',K_2')=\lk(K_2', K_3')=-1, \quad\lk(K_1', K_3')=1.
\]
We construct two branched surfaces $B_i$ in the exterior of $\mathcal{L}$, associated to the arcs $\alpha_i, \beta_i$ and $\gamma_i$, for $i=1, 2$, as described in Figure \ref{fig: family 1 exception}. It can be checked by direct inspection that for $i=1, 2$ the complement of $\alpha_i \cup \beta_i \cup \gamma_i$ contains no disc components in $S$, and that there are no sink discs. Hence we can apply Proposition \ref{prop: no sink implica taut foliation} and deduce that these branched surfaces carry laminations that extend to taut foliations on the manifolds obtained by Dehn filling the boundary tori along the multislopes realised by the boundary train tracks. 
First, we consider the boundary train tracks of $B_1$. Notice that the one contained in the boundary component of $M$ labelled with $1$ does not satisfy the hypothesis of Lemma \ref{lemma: unlinked}, and so we cannot use it to compute the slopes it realises. Nonetheless, a direct computation with weight systems shows that it realises all slopes in $(-\infty, 1)$. The two other boundary train tracks of $B_1$ are of type $b)$ and type $c)+d)$ and so realise all slopes in $(0, +\infty)$ and $(-\infty,0)$ on the corresponding boundary components. Summing up, the boundary train tracks of $B_1$ realise all the multislopes in $(-\infty, 1)\times(0, +\infty)\times(-\infty, 0)$.
The boundary train tracks associated to $B_2$ are all of type $a)+d)$ and hence realise all multislopes $(-\infty, 1)^3$. 
In particular, we have taut foliations on $S^3_{r,s,-1}(\mathcal{L})=S^3_{r-1,s+1}(L)$ for all $(r,s)\in (-\infty, 1)\times \matR$.
\end{itemize}

\begin{figure}[]
    \centering
    \includegraphics[width=0.82\textwidth]{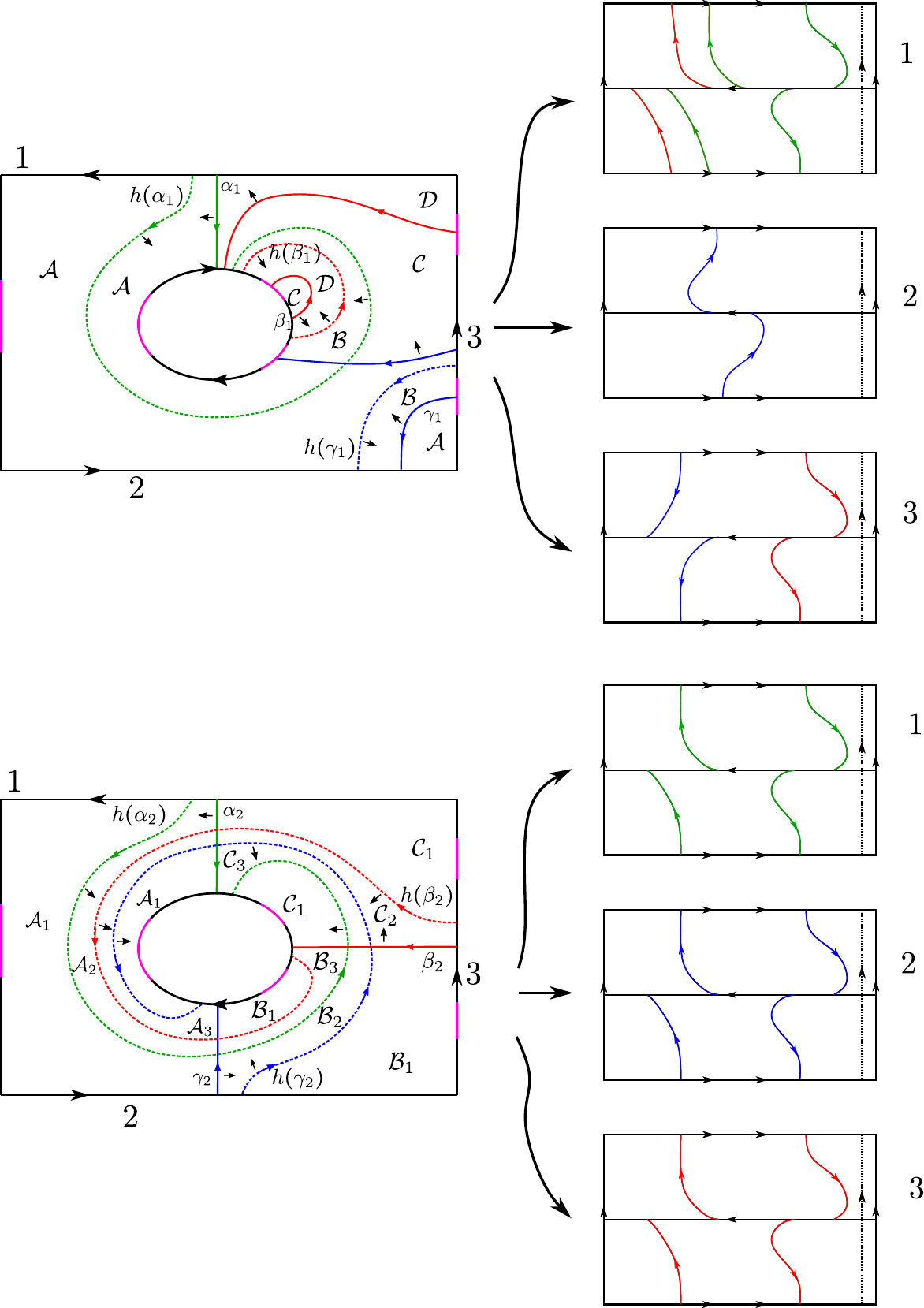}
    \caption{The arcs $\alpha_i, \beta_i, \gamma_i$ and their images via the monodromy $h$, together with the cusp directions of the associated branched surfaces and their boundary train tracks. We have also indicated the sectors of the branched surfaces in $S$.}
    \label{fig: family 1 exception}
\end{figure}

Since $L$ is symmetric, and since all multislopes in $(-1, +\infty)^2$ are covered by Lemma \ref{lemma: bridge twists}, the statement follows.
\end{proof}

\subsection{Study of the links in Family 4}\label{subsec: family 4}
We now focus our attention on the links of \emph{Family 4}, i.e. on the links of the form $L=L(2b_1,-2, 2b_3, \dots, -2, 2b_m)$ where exactly one $b_i$ is $-1$ and all the others are equal to $1$. 
We first study the case when $b_i=1$ for some $i\ne 1,m$. We write $m=2n+1$ for some positive integer $n$.
\begin{lemma}\label{lemma: riscrittura link}
Let $L=L(2b_1,-2, 2b_3, \dots, -2, 2b_m)$ where exactly one $b_{2k+1}$ is $-1$ and all the others are equal to $1$ and suppose that $2k+1\ne 1, m$. Then $L$ is isotopic as unoriented link to $L(-2k, -2, 2, -2, -2h)$, where $h=n-k$.
\end{lemma}

\begin{proof}
We will prove this algebraically. We start by computing the fraction associated to the link $L(-2k, -2, 2, -2, -2h)$. We have
\begin{align*}
-2k+\cfrac{1}{-2+\cfrac{1}{2+\cfrac{1}{-2-\cfrac{1}{2h}}}}&= -2k+\cfrac{1}{-2+\cfrac{1}{2-\cfrac{2h}{4h+1}}}=-2k+\cfrac{1}{-2+\cfrac{4h+1}{6h+2}}=\\
&=-2k+\frac{6h+2}{-(8h+3)}=\frac{16kh+6k+6h+2}{-(8h+3)}
\end{align*}
and this implies  $L(-2k, -2, 2, -2, -2h)=b(16kh+6k+6h+2, -(8h+3))$, where $b(p,q)$ denotes the two-bridge link associated to the rational $\frac{p}{q}$.

We now study the fraction corresponding to $L$. This fraction depends on $k$ and $n$, or equivalently on $k$ and $h=n-k$, and we denote its reduced representative by $\frac{\alpha_{k,h}}{\beta_{k,h}}$. Then we have

$$
\frac{\alpha_{k,h}}{\beta_{k,h}}=2+\cfrac{1}{-2+\cfrac{1}{2+\cfrac{1}{\ddots+\cfrac{1}{-2+\cfrac{1}{\textcolor{red}{-2}+\cfrac{q_h}{p_h}}}}}}
$$
where we have coloured the $-2$ corresponding to $2b_{2k+1}$, and where $\frac{p_h}{q_h}$ is defined in the following way
$$
\frac{p_h}{q_h}=\overbrace{-2+\cfrac{1}{2+\cfrac{1}{-2+\cfrac{1}{\ddots+\cfrac{1}{2}}}}}^{\text{length $2h$}}.
$$
One can check by induction on $h\geq 1$ that $\frac{p_h}{q_h}=\frac{2h+1}{-2h}$, and hence $p_h=2h+1$ and $q_h=-2h$. 
%In fact if $h=1$ we have that $\frac{p_1}{q_1}=-\frac{3}{2}$; using the inductive step we conclude that 
%\begin{align*}
%\frac{p_h}{q_h}&=-2+\cfrac{1}{2+\cfrac{q_{h-1}}{p_{h-1}}}=-2+\cfrac{p_{h-1}}{2p_{h-1}+q_{h-1}}=-\frac{3p_{h-1}+2q_{h-1}}{2p_{h-1}+q_{h-1}}=\\
%&=-\frac{6(h-1)+3-4(h-1)}{4(h-1)+2-2(h-1)}=\frac{2h+1}{-2h}.
%\end{align*}

We now prove by induction on $k$ that 
\begin{align*}
\alpha_{k,h}=16kh+6k+6h+2\\
\beta_{k,h}=16kh-2h+6k-1 
\end{align*}
for every $h$.
\begin{itemize}[leftmargin=*]
    \item \emph{Case $k=1$:} we have the equality
    $$
    \frac{\alpha_{1,h}}{\beta_{1,h}}=2+\cfrac{1}{-2+\cfrac{1}{-2+\cfrac{q_h}{p_h}}}=2+\cfrac{1}{-2-\cfrac{1+2h}{6h+2}}=2-\frac{6h+2}{14h+5}=\frac{22h+8}{14h+5}.
    $$
    Since $(22h+8,14h+5)= (q_h, p_h)=1$ we deduce that $\alpha_{1,h}=22h+8$ and $\beta_{1,h}=14h+5$.
    \item \emph{Case $k>1$:} we can use the following equality
    $$
    \frac{\alpha_{k,h}}{\beta_{k,h}}=2+\cfrac{1}{-2+\cfrac{\beta_{k-1,h}}{\alpha_{k-1,h}}}=2+\cfrac{\alpha_{k-1,h}}{-2\alpha_{k-1,h}+\beta_{k-1,h}}=\frac{3\alpha_{k-1,h}-2\beta_{k-1,h}}{2\alpha_{k-1,h}-\beta_{k-1,h}}
    $$
    and the fact that 
    $$(3\alpha_{k-1,h}-2\beta_{k-1,h},2\alpha_{k-1,h}-\beta_{k-1,h})=(\alpha_{k-1,h},\beta_{k-1,h})=1
    $$
    to deduce that $\alpha_{k,h}=3\alpha_{k-1,h}-2\beta_{k-1,h}$ and that $\beta_{k,h}=2\alpha_{k-1,h}-\beta_{k-1,h}$. Therefore we have
    \begin{align*}
    &\alpha_{k,h}-\beta_{k,h}=\alpha_{k-1,h}-\beta_{k-1,h}=8h+3\\
    &\alpha_{k,h}-\alpha_{k-1,h}=2(\alpha_{k-1,h}-\beta_{k-1,h})=16h+6.
    \end{align*}
    These equalities imply 
    \begin{align*}
    \alpha_{k,h}&=\alpha_{k-1,h}+16h+6=16kh+6k+6h+2
    \\
    \beta_{k,h}&=\alpha_{k,h}-8h-3=16kh-2h+6k-1
    \end{align*}
    and this proves the claim.
\end{itemize}
To conclude the proof of the lemma we just have to recall from Theorem \ref{teo: classification two-bridge links} that if $\beta'\equiv
\alpha+\beta \mod 2\alpha$ then the links $b(\alpha,\beta)$ and $b(\alpha,\beta')$ are isotopic after reversing the orientation of one of the components. In the case of our interest we have  $$\alpha_{k,h}+\beta_{k,h}\equiv-\alpha_{k,h}+\beta_{k,h}\equiv -(8h+3) \mod 2\alpha_{k,h}
$$
and this is exactly what we wanted.
\end{proof}

The description given by the previous lemma allows us to prove:

\begin{prop}
Let $L=L(2b_1,-2, 2b_3, \dots, -2, 2b_m)$ where exactly one $b_{2k+1}$ is $-1$ and all the others are equal to $1$ and suppose that $2k+1\ne 1, m$. Then all the rational Dehn surgeries on $L$ support co-orientable taut foliations.
\end{prop}

\begin{proof}
By virtue of Lemma \ref{lemma: riscrittura link} it is equivalent to study surgeries on links of the form $L_{k,h}=L(-2k, -2, 2, -2, -2h)$ where $h>0$ and $k>0$. These links can be obtained as surgeries on a $4$-component fibered link $\mathcal{L}$, as described in Figure \ref{fig:4-components plumbing}. 
Our aim now is to construct foliations on enough surgeries on $\mathcal{L}$.

\begin{figure}[]
    \centering
    \includegraphics[width=1\textwidth]{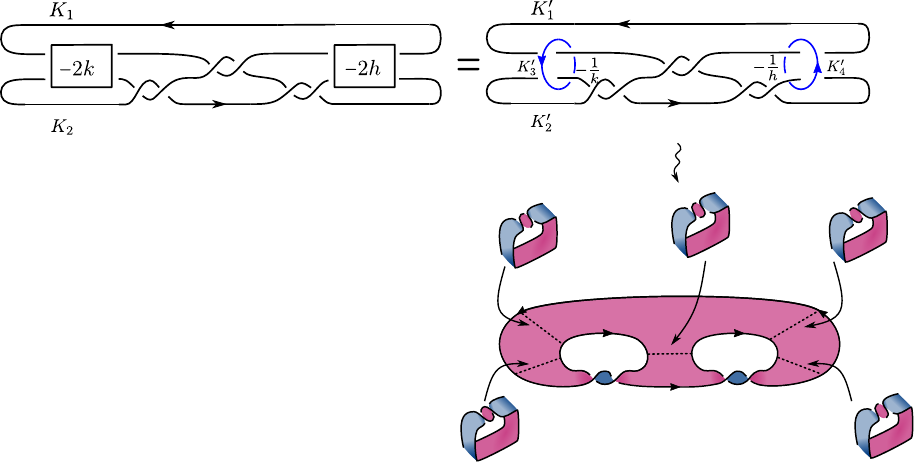}
    \caption{How to obtain the link $L(-2k,-2, 2,-2, -2h)$ as surgery on a $4$-component link $\mathcal{L}$. We also describe a fiber surface for $\mathcal{L}$, obtained as a sequence of Hopf plumbings.}
    \label{fig:4-components plumbing}
\end{figure}

The monodromy of the link $\mathcal{L}$ is given by $h=\tau_5\tau_3\tau_7\tau_6^{-1}\tau_4\tau_2\tau_1^{-1}$, where $\tau_i$ denotes the positive Dehn twist along the curve $c_i$ shown in Figure \ref{fig: 4-component twists}. If we label the components of $L$ and $\mathcal{L}$ as described in Figure \ref{fig:4-components plumbing}, the surgery coefficients change in the following way

\begingroup
\mathsurround=0pt
\begin{center}
    
\begin{tikzcd}
{\overbrace{(a, b, -\frac{1}{k}, -\frac{1}{h})}^{\text{Seifert framing for $\mathcal{L}$}}} \arrow[r, dashed] \arrow[d] & {\overbrace{(a, b)}^{\text{Seifert framing for $L$}}}              \\
{\underbrace{(a-1, b-1, -\frac{1}{k}, -\frac{1}{h}))}_{\text{Canonical framing for $\mathcal{L}$}}} \arrow[r]               & {\underbrace{(a-1+k+h, b-1+k+h)}_{\text{Canonical framing for $L$}}} \arrow[u].
\end{tikzcd}
\end{center}
\endgroup

As usual, when constructing foliations it is more natural to work with the framings given by the Seifert surfaces.

\begin{figure}[]
    \centering
    \includegraphics[width=0.5\textwidth]{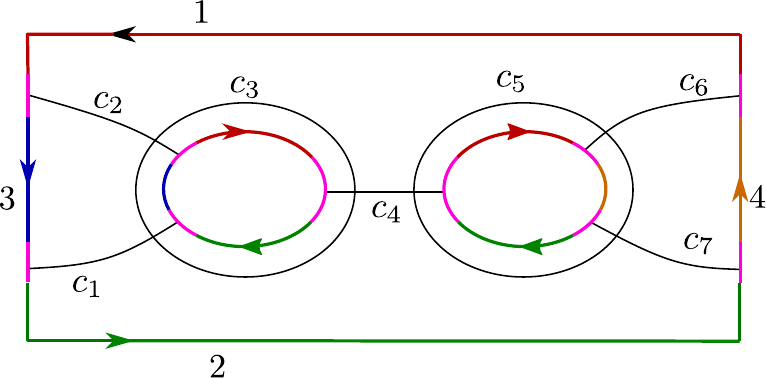}
    \caption{An abstract drawing of the fiber surface for the link $\mathcal{L}$, together with the curves $c_i$'s. We have also coloured the boundary components of $S$, the one labelled with $i$ corresponding to the component $K_i'$ of the link, for $i=1, 2, 3, 4$.}
    \label{fig: 4-component twists}
\end{figure}

We construct two branched surfaces in the exterior of $\mathcal{L}$. The first one is associated to the arcs $\alpha,\beta,\gamma,\delta,\epsilon$ depicted in Figure \ref{fig:4-components branched surface}, where we also describe the types of the boundary train tracks. The complement of these arcs in the fiber surface is not a disc (it is easier to see this by considering the complement of the images of these arcs via the diffeomorphism $h$) and the branched surface does not contain sink discs. In fact, there are five sectors in $S$, labelled with capital letters in Figure \ref{fig:4-components branched surface} and none of them is a sink disc. Therefore we can apply Proposition \ref{prop: no sink implica taut foliation} and deduce that there exist taut foliations on all the surgeries on $\mathcal{L}$ corresponding to multislopes in  $(0, +\infty)\times\matR\times(-\infty, 0)\times(-\infty,0)$.

\begin{figure}[]
    \centering
    \includegraphics[width=1\textwidth]{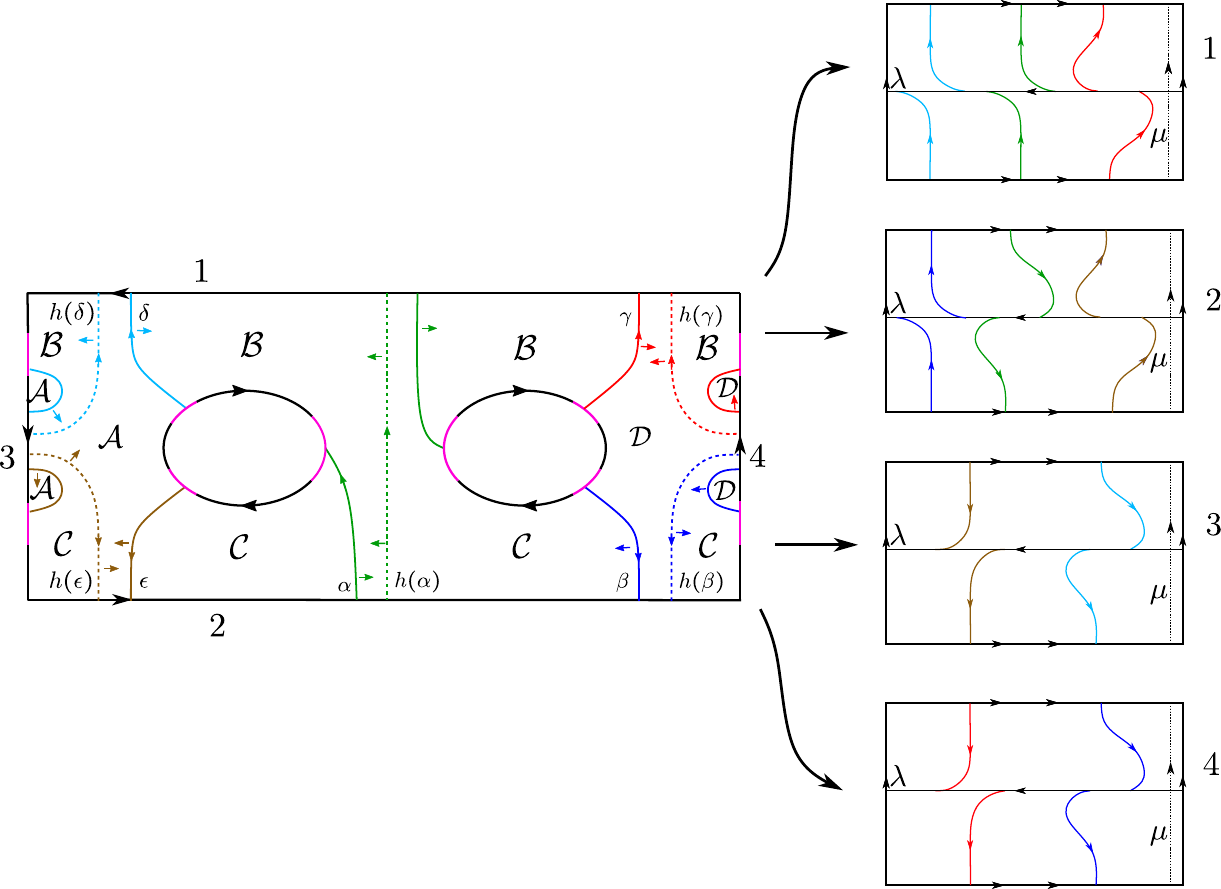}
    \caption{The arcs $\alpha,\beta,\gamma,\delta,\epsilon$ and the boundary train tracks of the associated branched surface.}
    \label{fig:4-components branched surface}
\end{figure}

The second branched surface is the one associated to the arcs described in Figure \ref{fig:4-components branched surface 1}. In this case all the boundary train tracks are of type $a)+d)$ and hence we are able to construct foliations on the surgeries corresponding to multislopes in $(-\infty, 1)^4$.

\begin{figure}[]
    \centering
    \includegraphics[width=0.65\textwidth]{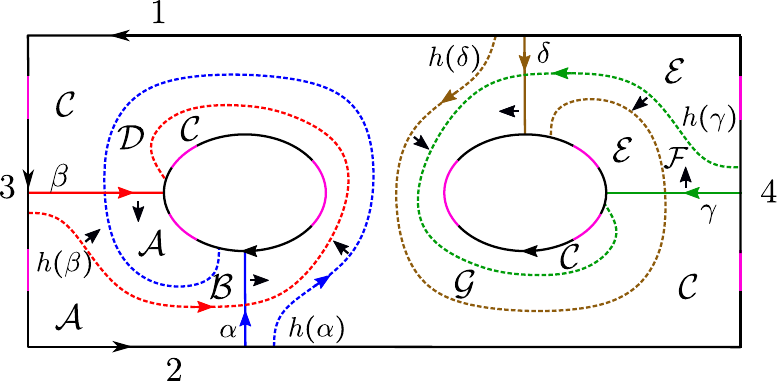}
    \caption{The arcs $\alpha, \beta, \gamma, \delta$ used to construct the second branched surface. We have also labelled the sector contained in $S$ and one can check that none of them is a sink disc.}
    \label{fig:4-components branched surface 1}
\end{figure}

This implies that for every $k>0$ and $h>0$ all the surgeries on the link $L_{k, h}$ corresponding to multislopes in $(0, +\infty)\times \matR$ and in $(-\infty, 1)^2$ support a co-orientable taut foliation. The conclusion follows using the fact that all these links are symmetric.  

\end{proof}

Now we only have to study the links $L=L(2b_1,-2, 2b_3, \dots, -2, 2b_{2n+1})$ where $b_1=-1$ and all the other $b_i$'s are $1$, or where $b_{2n+1}=-1$ and all the other $b_i$'s are $1$. The link $L(a_1, a_2, \dots, a_{2n+1})$ is isotopic to $L(a_{2n+1}, \dots, a_2, a_1)$, so we can reduce our study to the case when $b_{2n+1}=-1$ and we denote the corresponding link by $L_n$.

\begin{lemma} \label{lemma: altra rappresentazione di Ln}
The link $L_n$ is isotopic as unoriented link to the link $L(2, -2, -2n)$, illustrated in Figure \ref{figure:link L_n}.
\end{lemma}
\begin{proof}
We compute the fractions associated to these links. The one associated to $L(2, -2, -2n)$ is $\cfrac{6n+2}{4n+1}$. Therefore by Theorem \ref{teo: classification two-bridge links} the link $L(2, -2, -2n)$ is isotopic, after reversing the orientation of one of the components, to the link defined by the fraction $\cfrac{6n+2}{-(2n+1)}$.

The fractions $\cfrac{p_n}{q_n}$ associated to $L_n$ satisfy the following recursive equation
\begin{equation}\label{eq: recursive fraction}
\frac{p_n}{q_n}=2+\cfrac{1}{-2+\cfrac{q_{n-1}}{p_{n-1}}}=2+\cfrac{p_{n-1}}{-2p_{n-1}+q_{n-1}}=\cfrac{3p_{n-1}-2q_{n-1}}{2p_{n-1}-q_{n-1}}.
\end{equation}
Let us find an explicit formula for $p_n$ and $q_n$. It follows from Equation \eqref{eq: recursive fraction} that $$
p_n-q_n=p_{n-1}-q_{n-1}
$$
and as a consequence the quantity $p_i-q_i$ does not depend on the index $i$. Moreover, Equation \eqref{eq: recursive fraction} also implies 
$$
p_n-p_{n-1}=q_n-q_{n-1}=2(p_{n-1}-q_{n-1})
$$
and therefore also the quantity $p_i-p_{i-1}=q_i-q_{i-1}$ is constant in $i$. As when $n=1$ we have $\cfrac{p_1}{q_1}=\cfrac{8}{5}$, we deduce by induction that $p_n=8+(n-1)6=6n+2$ and $q_n=5+(n-1)6$. To conclude the proof is enough to observe that $-(2n+1)\equiv q_n^{-1} \mod 2p_n$ and use again Theorem \ref{teo: classification two-bridge links}.
\end{proof}

We will prove in the next section (see Proposition \ref{prop: Ln is L-space link}) that when $r\geq n$, $s\geq n$ the $(r,s)$-surgery on $L_n$ is an $L$-space, where the surgery coefficients are to be considered in the \emph{canonical} framing. We now prove that all the other (rational) surgeries on $L_n$ support co-orientable taut foliations.

\begin{prop}\label{prop: ctf on Ln}
Let $L_n=L(2b_1, -2, 2b_3, \dots, -2, 2b_{2n+1})$, where $b_{2n+1}=-1$ and all the other $b_i$'s are $1$, and let $r< n$, $s< n$ be rational numbers. Then the $(r,s)$-surgery on $L_n$ supports a co-orientable taut foliation, where the surgery coefficients are considered in the canonical framing.
\end{prop} 

\begin{proof}
We can suppose that $n\geq 2$, because $L_1$ is the Whitehead link, for which the corresponding result was proved in \cite{S}.  By Lemma \ref{lemma: altra rappresentazione di Ln} we have $L_n=L(2, -2, -2n)$ as unoriented links and by using this representation it is evident that $L_n=S^3_{\bullet, \bullet, -\frac{1}{n}}(\mathcal{L})$, where $\mathcal{L}$ is drawn in Figure \ref{fig: Ln plumbing}. This figure also shows a fiber surface $S$ for $\mathcal{L}$ obtained via a sequence of four Hopf plumbings.

\begin{figure}[]
    \centering
    \includegraphics[width=1\textwidth]{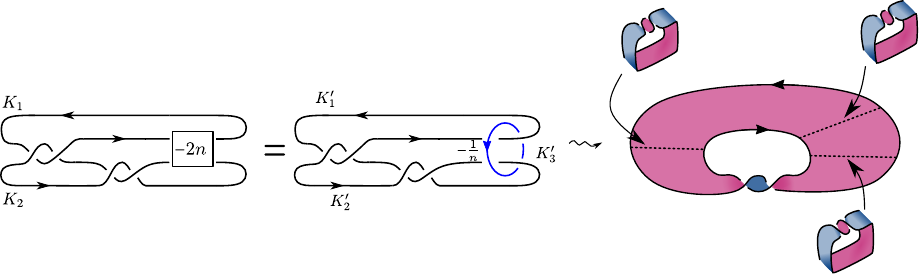}
    \caption{How to obtain the links $\{L_n\}_{n\geq 1}$ as surgery on a $3$-components link $\mathcal{L}$ and a fiber surface $S$ for $\mathcal{L}$.}
    \label{fig: Ln plumbing}
\end{figure}

We choose four triples $\alpha,\beta,\gamma$ of oriented arcs in $S$ and consider the four branched surfaces in the exterior of $\mathcal{L}$ associated to these arcs, as depicted in Figure \ref{fig: Ln branched surface_1} and Figure \ref{fig: Ln branched surface 2}. Each of these triples has the property that its complement in $S$ contains no disc components.

\begin{figure}[]
    \centering
    \includegraphics[width=0.8\textwidth]{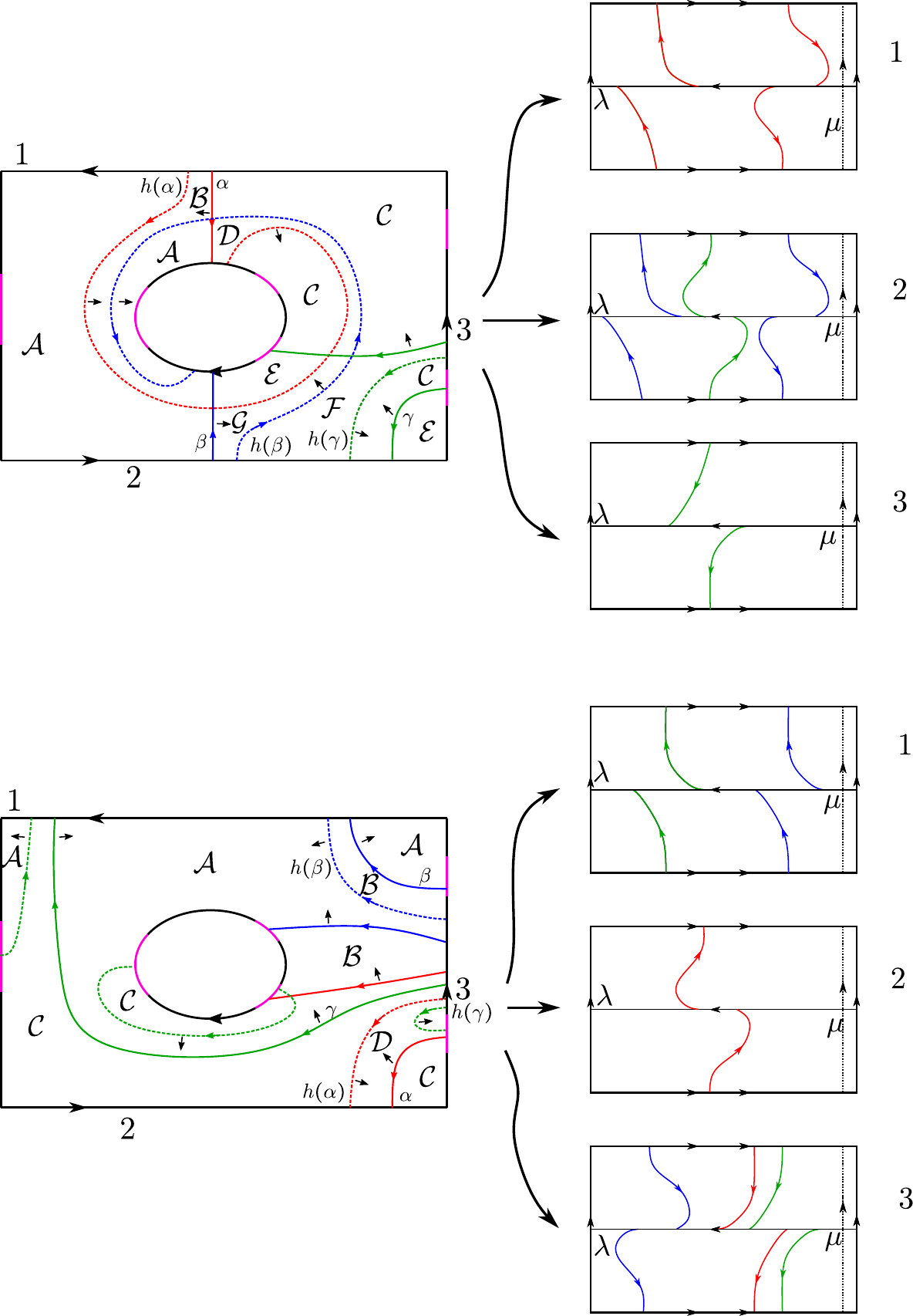}
    \caption{How to choose two of the four triples of arcs $\alpha,\beta, \gamma$. The picture also represents their images via the monodromy, the boundary train tracks, and the sectors in $S$.}
    \label{fig: Ln branched surface_1}
\end{figure}

Moreover, the figures also illustrate the sectors of the branched surfaces contained in $S$ and it can be checked that none of these is a sink disc. Thus, thanks to Proposition \ref{prop: no sink implica taut foliation} we only need to study the boundary train tracks of these branched surfaces in order to construct the desired taut foliations. The multislopes realised by these branched surfaces in the Seifert framing of $\mathcal{L}$ are, respectively:
\begin{itemize}
\item all the multislopes in $(-\infty, 1) \times \matR \times (-1, 0)$;
\item all the multislopes in $(0, 2) \times (0, +\infty) \times (-\infty, 0)$;
\item all the multislopes in $(0, 2) \times (-\infty, 0)\times (-1, 0)$;
\item all the multislopes in $(-\infty, 2) \times (-1, 1) \times (-1, 0)$.
\end{itemize}

\begin{figure}[]
    \centering
    \includegraphics[width=0.8\textwidth]{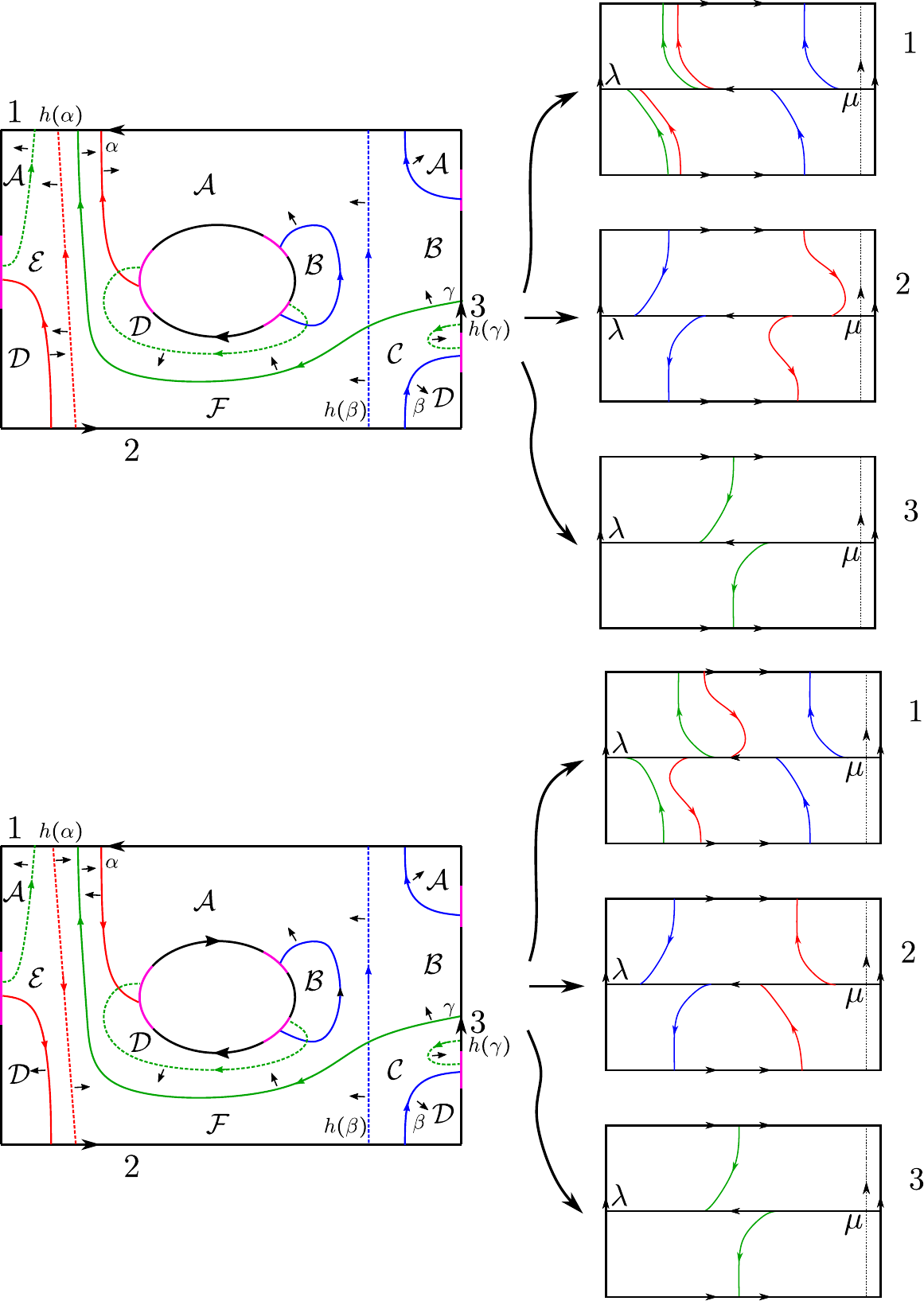}
    \caption{How to choose the two other triples of arcs $\alpha,\beta, \gamma$. The picture also represents their images via the monodromy, the boundary train tracks, and the sectors in $S$.}
    \label{fig: Ln branched surface 2}
\end{figure}

We now prove that by considering $L_n$ as $-\frac{1}{n}$ surgery on the third component of $\mathcal{L}$, we have constructed the desired foliations on the surgeries on $L_n$. First of all we observe that  
$$
\lk(K_1',K_2')=\lk(K_1', K_3')=1, \quad\lk(K_2', K_3')=-1
$$
and by using formula \eqref{eq: cambio coefficienti} we deduce the following change of surgery coefficients:
$$
{\overbrace{(a, b, -\frac{1}{n})}^{\text{Seifert framing for $\mathcal{L}$}}} \rightarrow \overbrace{(a-2, b, -\frac{1}{n})}^{\text{Canonical framing for $\mathcal{L}$}} \rightarrow \overbrace{(a+n-2, b+n)}^{\text{Canonical framing for $L_n$}}.
$$
Therefore, for every $n\geq 2$, we obtain taut foliations on all the surgeries on $L_n$ corresponding to multislopes in 

\begin{itemize}
\item  $W=(-\infty, n-1) \times \matR $;
\item  $X=(n-2, n) \times (n, +\infty)  $;
\item  $Y=(n-2, n) \times (-\infty, n)$;
\item  $Z=(-\infty, n) \times (n-1, n+1)$.
\end{itemize}
We now show that these four sets are enough to deduce that, for all $n\geq 2$, all the surgeries on $L_n$ corresponding to multislopes $(r_1, r_2)$ where $r_1 < n$ or $r_2 < n$ support a co-orientable taut foliation. In fact suppose that we have such a pair $(r_1, r_2)$. Since $L_n$ is symmetric we can suppose that $r_1 < n$ and we have the following cases:
\begin{itemize}
\item \emph{$r_1< n-1$:} in this case the pair is contained in the set $W$;
\item \emph{$n-1 \leq r_1 < n$:} if $r_2 > n$ the pair is contained in $X$, if $r_2<n$ we conclude by using the set $Y$ and if $r_2=n$ we use the set $Z$.
\end{itemize}
This concludes the proof.
\end{proof}

\section{L-spaces}\label{sec: L-spaces}

We now study the $L$-space surgeries on the links $L_n$, illustrated in Figure \ref{figure:link L_n}, and conclude the proof of Theorem \ref{thm: main theorem} and Proposition \ref{prop: link L_n}. To do this we will use the main result of \cite{RR}. We recall some definitions and fix some notation.

Let $Y$ be a rational homology solid torus, i.e. $Y$ is a compact oriented $3$-manifold with toroidal boundary such that $H_*(Y;\matQ)\cong H_*(\matD^2\times S^1; \matQ)$.

We are interested in the study of the Dehn fillings on $Y$. We define the \emph{set of slopes in Y} as
$$
Sl(Y)=\{\alpha\in H_1(\partial Y;\matZ)|\, \alpha \text{ is primitive}\}/\pm 1.
$$
We will denote with $Y(\alpha)$ the Dehn filling of $Y$ associated to $[\alpha]\in Sl(Y)$.

Notice that as a consequence of $Y$ being a rational homology solid torus, there is a distinguished slope in $Sl(Y)$ that we call the \emph{homological longitude} of $Y$ and that is defined in the following way. We denote with $i: H_1(\partial Y;\matZ)\rightarrow H_1(Y;\matZ)$ the map induced by the inclusion $\partial Y\subset Y$ and we consider a primitive element $l\in H_1(\partial Y;\matZ)$ such that $i(l)$ is torsion in $H_1(Y;\matZ)$. The element $l$ is unique up to sign, and its equivalence class $[l]\in Sl(Y)$ is the homological longitude of $Y$. This definition, which may seem counterintuitive, is given so that when $Y$ is the complement of a knot in $S^3$, the homological longitude of $Y$ coincides with the slope defined by the longitude of the knot.
\newline

We want to study the fillings on $Y$ that are $L$-spaces. 
For this reason we define the set of the \emph{$L$-space filling slopes}:
$$
\mathcal{L}(Y)=\{[\alpha]\in Sl(Y)|\,\, Y(\alpha)\text{ is an $L$-space}\}.
$$
Once we fix a basis $(\mu,\lambda)$ for $H_1(\partial Y; \matZ)$ we can identify the set $Sl(Y)$ with $\overline{\matQ}$. The following theorem is a straightforward consequence of \cite[Theorem~1.6]{RR}.

\begin{teo}\label{thm:RR}
Let $Y$ be a rational homology solid torus and let $[\alpha]\ne [\beta]$ be two slopes in $\mathcal{L}(Y)$. Then $\mathcal{L}(Y)$ contains the interval in $Sl(Y)$ between $[\alpha]$ and $[\beta]$ that does not contain the homological longitude $[l]$. \qed
\end{teo}

We want to use this result to study $L$-space surgeries on links. If $L$ is a link in $S^3$, we denote by $\mathcal{L}(L)$ the set of slopes in the exterior of $L$ such that the corresponding surgery is an $L$-space. For each component of the link we fix the canonical meridian and longitude, and in this way we can identify $\mathcal{L}(L)$ with a subset of $\overline{\matQ}^d$, where $d$ is the number of components of the link.
We fix $d=2$, i.e. we suppose that $L$ has two components $K_1$ and $K_2$. Given $(r_1,r_2)\in \overline{\matQ}^2$, we denote by 
\begin{itemize}
    \item $S^3_{r_1,r_2}(L)$ the $(r_1,r_2)$-surgery on $L$;
    \item $S^3_{r_1, \bullet}(L)$ the manifold obtained by drilling $K_2$ and performing $r_1$-surgery on $K_1$;
    \item $S^3_{\bullet,r_2}(L)$ the manifold obtained by drilling $K_1$ and performing $r_2$-surgery on $K_2$.
\end{itemize}
Recall that if $L$ has two components, then by using Mayer-Vietoris one can see that the manifold $S^3_{r_1,r_2}(L)$ is not a rational homology sphere if and only if $\{r_1, r_2\}=\{0, \infty\}$ or $r_1r_2=\lk(L)^2$, where $\lk(L)$ denotes the linking number of the components of $L$. Hence if $r_1\ne 0$ the manifold $S^3_{r_1, \bullet}(L)$ is a rational homology solid torus with homological longitude given by $\frac{\lk(L)^2}{r_1}\in \matQ$. Analogously, if $r_2\ne 0$ the manifold $S^3_{\bullet, r_2}(L)$ is a rational homology solid torus with homological longitude given by $\frac{\lk(L)^2}{r_2}\in \matQ$.

\begin{prop}\label{prop: Link with unknotted components}
Let $L$ be a two-component link with two unknotted components. Suppose that $(r_1,r_2)\in \mathcal{L}(L)$ with $r_1r_2>\lk(L)^2$ and $r_1>0,r_2>0$. Then $\big{(}[r_1,\infty]\times[r_2,\infty]\big{)}\cap \overline{\matQ}^2$ is contained in $\mathcal{L}(L)$. Analogously, if $r_1r_2>\lk(L)^2$ and $r_1<0,r_2<0$ then $\big{(}[\infty,r_1]\times[\infty,r_2]\big{)}\cap \overline{\matQ}^2$ is contained in $\mathcal{L}(L)$.
\end{prop}

\begin{proof}
The proof is the straightforward adaptation of the proof of \cite[Lemma~2.6]{S}. We report it here for convenience of the reader. We prove the proposition in the case $r_1r_2>\lk(L)^2$ and $r_1>0, r_2>0$. The other case is analogous. 
We consider the manifold $Y=S^3_{r_1,\bullet}$. We have that $r_2\in \mathcal{L}(Y)$ and since the components of $L$ are unknotted it follows that also $\infty\in \mathcal{L}(L)$. In fact $S^3_{r_1,\infty}(L)$ is a lens space, and hence an $L$-space. Thus we can deduce, by virtue of Theorem \ref{thm:RR}, that the interval between $r_2$ and $\infty$ that does not contain the homological longitude is contained in $\mathcal{L}(Y)$. The homological longitude $\frac{\lk(L)^2}{r_1}$ is smaller than $r_2$, so we deduce that $[r_2, \infty]\cap\overline{\matQ}\subset \mathcal{L}(Y)$. In other words we have proved that $S^3_{r_1,s}(L)$ is an $L$-space for all $s\geq r_2$. Now we fix $s\geq r_2$ and consider the manifold $Y_s=S^3_{\bullet,s}$. As a consequence of $r_1$ and $\infty$ belonging to $\mathcal{L}(Y_s)$, we can apply again Theorem \ref{thm:RR} and deduce that the interval between $r_1$ and $\infty$ that does not contain the homological longitude is contained in $\mathcal{L}(Y_s)$. Since $r_1\geq \frac{\lk(L)^2}{r_2}\geq \frac{\lk(L)^2}{s}$ and the last term in the chain of inequalities is the homological longitude of $Y_s$, we conclude that $[r_1,\infty]\cap \overline{\matQ}\subset \mathcal{L}(Y_s)$ for all $s\geq r_2$. This is exactly equivalent to saying that $\big{(}[r_1,\infty]\times [r_2,\infty]\big{)}\cap\overline{\matQ}^2\subset \mathcal{L}(L)$. A pictorial sketch of the proof is described in Figure \ref{figure:proof}.
\end{proof}

\begin{figure}[]
    \centering
    \includegraphics[width=1\textwidth]{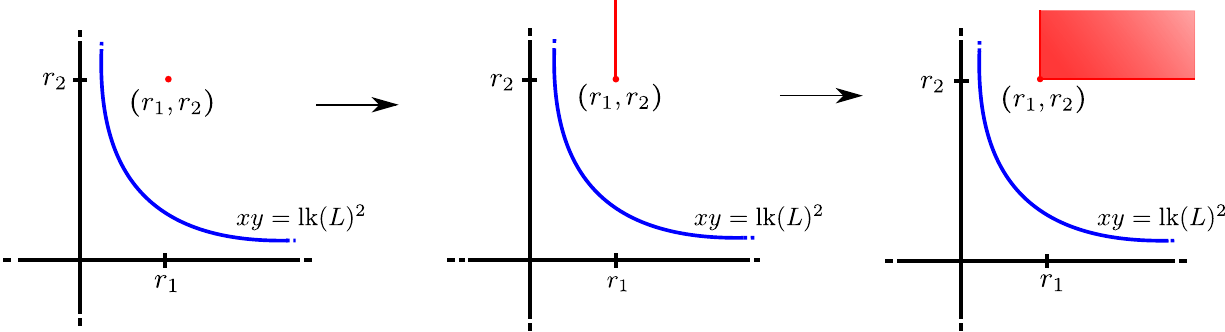}
    \caption{A pictorial sketch of the proof.}
    \label{figure:proof}
\end{figure}

We are now ready to prove the main result of this section. Recall that since two-bridge links have unknotted components, the rational homology spheres associated to surgery coefficients $(r,s)$ where at least one of $r$ and $s$ is $\infty$ are all $L$-spaces. For this reason, we will only consider the set $\mathcal{L}(L_n)\cap \matQ^2$.

\begin{prop}\label{prop: Ln is L-space link}
Let $L_n$ the link described in Figure \ref{figure:link L_n}. Then $\mathcal{L}(L_n)\cap \matQ^2=[n,+\infty)^2\cap \matQ^2$ .
\end{prop}
\begin{proof}
Since $L$-space do not support taut foliations, it follows from Proposition \ref{prop: ctf on Ln} that it suffices to prove that  
$$
[n,+\infty)^2\subset \mathcal{L}(L_n).
$$
The link $L_n$ satisfies $\lk(L_n)^2=(n-1)^2$ and its components are unknotted, hence by Proposition \ref{prop: Link with unknotted components} it is enough to prove that $(n,n)\in \mathcal{L}(L_n)$.
We can see the links $L_n$ as surgeries on a three-component link $\mathscr{L}$, as represented in Figure \ref{figure:Ln+L}. We have also fixed an orientation of this link, that we will use later in the proof.

\begin{figure}[]
    \centering
    \includegraphics[width=0.4\textwidth]{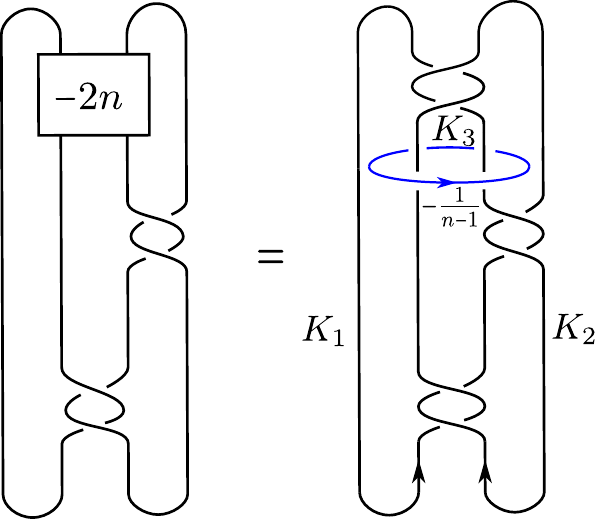}
    \caption{How to obtain the links $\{L_n\}_{n\geq 1}$ as surgeries on a $3$-component link $\mathscr{L}$.} 
    \label{figure:Ln+L}
\end{figure}
More precisely we have that $S^3_{a,b,-\frac{1}{n-1}}(\mathscr{L})=S^3_{a+n-1,b+n-1}(L_n)$. This implies that the statement is equivalent to proving that $S^3_{1,1,-\frac{1}{n-1}}(\mathscr{L})$ is an $L$-space for all $n\geq 1$ and to prove this we will apply Theorem \ref{thm:RR} to the rational homology solid torus $S^3_{1,1,\bullet}(\mathscr{L})$. Denoting this manifold by $Y$, we have:
\begin{itemize}[leftmargin=*]
    \item \emph{$\infty\in \mathcal{L}(Y)$:} in fact $S^3_{1,1,\infty}(\mathscr{L})$ is $(1,1)$-surgery on the Whitehead link. This is the Poincaré homology sphere and manifolds with finite fundamental group are $L$-spaces \cite[Proposition~2.2]{OS1};
    \item \emph{$1\in \mathcal{L}(Y)$:} in fact $S^3_{1,1,1}(\mathscr{L})$ is $(0,0)$-surgery on the Hopf link, see Figure \ref{figure:Hopf}. This manifold is $S^3$ and therefore an $L$-space;
    \item \emph{the homological longitude of $Y$ is the slope $2$:} to prove this we have to do a simple computation. We fix an orientation for the link and we denote the components of $\mathscr{L}$ with $K_1$, $K_2$ and $K_3$ as in Figure \ref{figure:Ln+L}.

    We have that $\lk(K_1,K_3)=\lk(K_2, K_3)=1$ and  $\lk(K_1,K_2)=0$. Consequently, a presentation matrix for $H_1(S^3_{1,1,\frac{p}{q}}(\mathscr{L}),\matZ)$ is given by
    $$
    A=\begin{pmatrix}
    1 &0 &q \\
    0 &1 &q \\
    1 &1 &p
    \end{pmatrix}
    $$
    and in particular $S^3_{1,1,\frac{p}{q}}(\mathscr{L})$ is not a rational homology sphere if and only if the determinant of $A$ is zero. This happens if and only if $p=2q$ and therefore $2$ is the homological longitude of the manifold $S^3_{1,1,\bullet}(\mathscr{L})$.
\end{itemize}
What we have proved implies by Theorem \ref{thm:RR} that $(-\infty, 1]\cap \matQ\subset \mathcal{L}(Y)$. In particular  $S^3_{1,1,-\frac{1}{n-1}}(\mathscr{L})$ is an $L$-space for all $n\geq 1$ and this manifold is exactly the $(n,n)$-surgery on $L_n$.
\end{proof}

\begin{figure}[]
    \centering
    \includegraphics[width=0.6\textwidth]{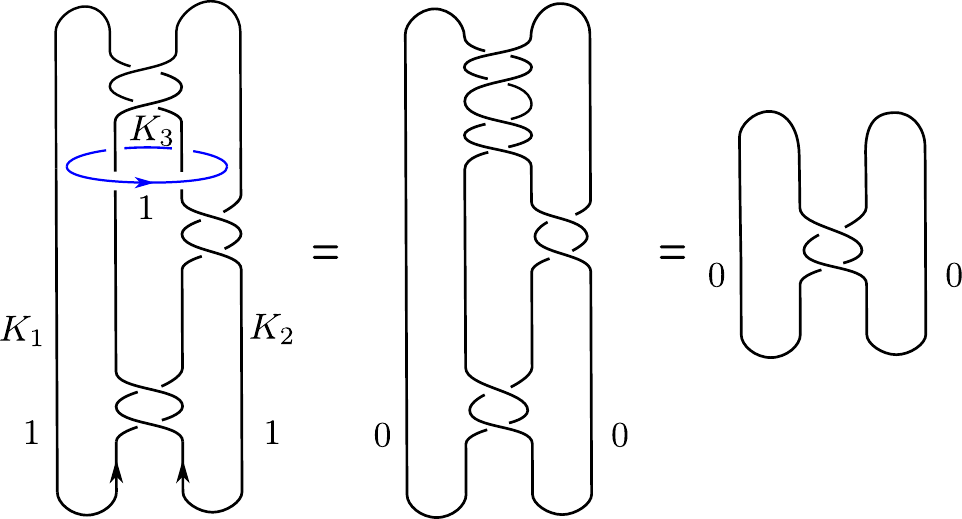}
    \caption{The $(1, 1, 1)$-surgery on $\mathscr{L}$ is $(0,0)$-surgery on the Hopf link} 
    \label{figure:Hopf}
\end{figure}
\begin{rem}
In the terminology of \cite{GN}, the links $L_n$ are \emph{$L$-space links}. In \cite{Liu1,Liu}, Liu conjectured that a two-bridge link is an $L$-space link if and only if is of the form $b(pq-1,-q)$, where $p$ and $q$ are odd positive integers. This conjecture was proved by Dawra in \cite{Dawra}. It is not difficult to prove that the link $L_n$, as an unoriented link, is isotopic to $b(6n+2,-3)$.
\end{rem}

%%%%%%%%%%%%%%%%%%%%   End of main body of article
%
%                             References
%
%   BiBTeX users uncomment the following line:
%
\bibliographystyle{gtart}
%

%\begin{thebibliography}
\bibliography{Bib}
%\end{thebibliography}

\end{document}